\documentclass[hidelinks,onefignum,onetabnum,twoside]{neutral}

\usepackage{amsfonts}
\usepackage{epstopdf}
\usepackage{algorithmic}
\usepackage{comment}
\ifpdf
  \DeclareGraphicsExtensions{.eps,.pdf,.png,.jpg}
\else
  \DeclareGraphicsExtensions{.eps}
\fi
\graphicspath{ {images/} }

\newneutralremark{remark}{Remark}
\newneutralremark{hypothesis}{Hypothesis}
\crefname{hypothesis}{Hypothesis}{Hypotheses}
\newneutralthm{claim}{Claim}

\headers{The Fascinating World of $2 \times 2 \times 2$ Tensors}{G. H. Brown, J. Kileel, and T. G. Kolda}

\title{The Fascinating World of 2 $\bm\times$ 2 $\bm\times$ 2 Tensors:\\
  Its Geometry and Optimization Challenges%
  \thanks{
    \funding{G.B. was supported in part by a CSEM fellowship from the Oden Insitute.
J.K. and G.B. were  supported in part by 
NSF DMS 2309782, NSF DMS 2436499, NSF
CISE-IIS 2312746 and DE SC0025312.
T.K. was supported in part by the JTO visitor program at the Oden Institute.}
  }
}

\author{
  Gabriel H. Brown\thanks{Oden Institute, University of Texas at Austin, Austin, TX
  (\email{ghbrown@utexas.edu}).}
  \and Joe Kileel\thanks{Department of Mathematics and Oden Institute, University of Texas at Austin, Austin, TX (\email{jkileel@math.utexas.edu}).}
\and Tamara G. Kolda\thanks{MathSci.ai, Dublin, CA
  (\email{tammy.kolda@mathsci.ai}).}
}

\usepackage{amsopn}
\usepackage{amsmath}
\usepackage{amssymb}
\usepackage{amsbsy}
\usepackage{bm}
\usepackage{mathrsfs}
\usepackage{mathtools}
\usepackage[mathscr]{eucal}
\usepackage{enumitem}
\usepackage{esint}
\usepackage{stmaryrd}
\usepackage{xspace}
\usepackage{xparse}
\usepackage{epsfig}
\usepackage{scrextend}
\usepackage{float}
\usepackage{tikz}
\usepackage{pgfplots}
\pgfplotsset{compat=1.18}
\usepgfplotslibrary{fillbetween,groupplots}
\usetikzlibrary{math,calc}
\usepackage{tabularray}
\usepackage{braket}  %
\usepackage{hyperref}
\usepackage{xcolor}
\usepackage{fixme}
\DeclarePairedDelimiter{\prn}{(}{)}  %
\DeclarePairedDelimiter{\abs}{\lvert}{\rvert}  %
\DeclarePairedDelimiter{\norm}{\lVert}{\rVert} %
\usepackage{latexsym} %

\usepackage{subcaption}

\makeatletter
\let\oldabs\abs
\def\abs{\@ifstar{\oldabs}{\oldabs*}}
\let\oldnorm\norm
\def\norm{\@ifstar{\oldnorm}{\oldnorm*}}
\makeatother

\makeatletter
\renewcommand*\env@matrix[1][*\c@MaxMatrixCols c]{%
  \hskip -\arraycolsep
  \let\@ifnextchar\new@ifnextchar
  \array{#1}}
\makeatother

\newneutralremark{example}{Example}
\newneutralremark{exercise}{Exercise}

\SetSymbolFont{stmry}{bold}{U}{stmry}{m}{n}

\DeclareMathOperator{\argmin}{argmin}

\DeclareMathOperator{\rank}{rank}

\DeclareMathOperator{\mrank}{multirank}

\renewcommand{\vec}{\operatorname{vec}}
\DeclareMathOperator{\col}{col}
\newcommand{\ttt}{2 \times 2 \times 2}

\newcommand{\mbb}{\mathbb}
\newcommand{\mbf}{\mathbf}

\NewDocumentEnvironment{te}{}{\begin{bmatrix}[r r | r r]}{\end{bmatrix}}
\NewDocumentEnvironment{tensoreight}{ob}{
  \IfValueTF{#1}
  {\resizebox{!}{#1}{$\begin{te} #2 \end{te}$}}
  {\begin{te} #2 \end{te}}
}{}

\newcommand{\tabletensor}[8]{
  \resizebox{0.15\textwidth}{!}{
    $
    \begin{tensoreight}
      #1 & #2 & #3 & #4 \\ #5 & #6 & #7 & #8
    \end{tensoreight}
    $
  }
}

\NewDocumentCommand{\Tn}{O{} m !g}{%
  \boldsymbol{#1{\mathscr{\MakeUppercase{#2}}}}%
  \IfValueT{#3}{_{#3}}%
}

\NewDocumentCommand{\Tr}{s}{\IfBooleanTF{#1}{\vphantom{\intercal}}{\intercal}}

\NewDocumentCommand{\Mx}{O{} m !g t' t"}{
  \bm{#1{\mathbf{\MakeUppercase{#2}}}}%
  \IfValueT{#3}{_{#3}}%
  \IfBooleanTF{#4}{^{\Tr}}{%
    \IfBooleanT{#5}{^{\Tr*}}}%
}

\NewDocumentCommand{\Tm}{m m}{
  \Mx{#1}{(#2)}
}

\NewDocumentCommand{\Vc}{O{} m !g t' t"}{
  \bm{#1{\mathbf{\MakeLowercase{#2}}}}%
  \IfValueT{#3}{_{#3}}%
  \IfBooleanTF{#4}{^{\Tr}}{%
    \IfBooleanT{#5}{^{\Tr*}}}%
}

\NewDocumentCommand{\Real}{}{\mathbb{R}}
\NewDocumentCommand{\T}{}{\Tn{T}}
\NewDocumentCommand{\TM}{m}{\Mx{T}{(#1)}}
\NewDocumentCommand{\G}{}{\Tn{G}}
\NewDocumentCommand{\A}{}{\Mx{A}}
\NewDocumentCommand{\B}{}{\Mx{B}}
\NewDocumentCommand{\C}{}{\Mx{C}}
\NewDocumentCommand{\U}{}{\Mx{U}}
\NewDocumentCommand{\V}{}{\Mx{V}}
\NewDocumentCommand{\W}{}{\Mx{W}}

\NewDocumentCommand{\GL}{D(){n}}{\operatorname{GL}(#1)}
\NewDocumentCommand{\SL}{D(){n}}{\operatorname{SL}(#1)}
\RenewDocumentCommand{\O}{D(){n}}{\operatorname{O}(#1)}
\NewDocumentCommand{\SO}{D(){n}}{\operatorname{SO}(#1)}
\NewDocumentCommand{\NV}{D(){n}}{\operatorname{NV}(#1)}
\NewDocumentCommand{\va}{}{\mbf{a}}
\NewDocumentCommand{\vb}{}{\mbf{b}}
\NewDocumentCommand{\vc}{}{\mbf{c}}

\definecolor{tri_green}{HTML}{a2f9bb}  %
\definecolor{tri_blue}{HTML}{a2cff9}
\definecolor{tri_red}{HTML}{f9a2ac}
\definecolor{sunset_yellow}{HTML}{ffa600} %
\definecolor{sunset_red}{HTML}{ff6361}
\definecolor{sunset_blue}{HTML}{58508d}
\definecolor{sunset_purple}{HTML}{bc5090}
\colorlet{light}{sunset_yellow}
\colorlet{mid}{sunset_red}
\colorlet{dark}{sunset_blue}
\colorlet{extra}{sunset_purple}  %

\usetikzlibrary{arrows.meta}
\usetikzlibrary{shapes.misc}
\usetikzlibrary{positioning}
\tikzset{cross/.style={cross out, draw=black, minimum size=2*(#1-\pgflinewidth), inner sep=0pt, outer sep=0pt},
  cross/.default={1pt}}

\fxsetup{
  status=final,
  marginface=\linespread{1}\tiny,
  theme=color,
  silent,
}

\FXRegisterAuthor{gb}{agb}{\colorbox{blue!50}{\color{black}GB}}
\FXRegisterAuthor{jk}{ajk}{\colorbox{green!50}{\color{black}JK}}
\FXRegisterAuthor{tk}{atk}{\colorbox{red!50}{\color{black}TK}}

\newcommand{\webtool}[1]{\href{https://ghbrown.net/deter#1}{\texttt{https://ghbrown.net/deter#1}}}

\ifpdf
\hypersetup{
  pdftitle={The Fascinating World of 2 x 2 x 2 Tensors},
  pdfauthor={G. H.  Brown, J. Kileel, and T. G. Kolda}
}
\fi

\begin{document}
\maketitle

\begin{abstract}
This educational article highlights the geometric and algebraic complexities that distinguish tensors from matrices,
to supplement coverage in advanced courses on linear algebra, matrix analysis, and tensor decompositions.
Using the case of real-valued $2 \times 2 \times 2$ tensors,
we show how tensors violate many well-known properties of matrices:
\begin{itemize}[leftmargin=2em]
    \item[--] The rank of a matrix is bounded by its smallest dimension,
    but a $2 \times 2 \times 2$ tensor can be rank 3.
    \item[--] Matrices have a single typical rank, but the rank of a generic $2 \times 2 \times 2$ tensor can be 2 \emph{or} 3---it has two typical ranks.
    \item[--] Any limit point of a sequence of matrices of rank $r$ is at most rank $r$, but a limit point of a sequence of $2 \times 2 \times 2$  tensors of rank 2 can be rank 3 (a higher rank).
    \item[--] Matrices always have a best rank-$r$ approximation,
    but \emph{no} rank-3 tensor of size $2 \times 2 \times 2$ has a best rank-2 approximation.
\end{itemize}
We unify the analysis of the matrix and tensor cases using tools from algebraic geometry and optimization, providing derivations of these surprising facts. To build intuition for the
geometry of rank-constrained sets, students and educators can explore the geometry of matrix and tensor ranks via our interactive visualization tool.
\end{abstract}

\begin{keywords}
  tensors, matrices, 
  rank, geometry,
  hyperdeterminant,
  group action,
  ill-posedness,
  canonical forms,
  semi-algebraic sets, 
  geometry, 
  border rank,
  visualization
\end{keywords}

\begin{MSCcodes}
  	15A69, 65F99, 14P10, 97N30, 97H60
\end{MSCcodes}

\section{Introduction}
\label{sec:introduction}

Tensors have emerged as powerful tools for modeling complex relationships across multiple dimensions.
In this educational article, we aim to highlight key
characteristics that distinguish tensor analysis from matrix analysis, with a focus the geometry of sets of matrices and tensors of particular ranks.
We have created an interactive tool that students and educators can use to visualize and build intuition for the results in this paper.
This article provides material suitable for a numerical linear algebra, data science, or first algebraic geometry course.
It is peppered with helpful examples and exercises designed to deepen understanding.
We assume some knowledge of linear algebra and use examples based on linear algebra to motivate the discussion.

One characterization of a \emph{rank-1 matrix} is as the outer product of two vectors, which we denote here as $\va \circ \vb$, with $(i,j)$-entry equal to $a_i b_j$.
The \emph{rank} of a matrix $\Mx{M} \in \Real^{m \times n}$ is the minimum $r$ such that
$\Mx{M}$ can be written as the sum of $r$ rank-1 matrices:
\begin{displaymath}
  \Mx{M}
  = \A \B'
  \equiv \sum_{\ell=1}^r \mbf{a}_{\ell} \circ \mbf{b}_{\ell}
  .
\end{displaymath}
Here, $\A \in \Real^{m \times r}$ and $\B \in \Real^{n \times r}$ are
the \emph{factor matrices} and $\mbf{a}_{\ell}$ and $\mbf{b}_{\ell}$
are the \emph{factors} representing the $\ell$th columns of $\A$ and
$\B$, respectively. This is illustrated in \cref{fig:matrix-decomp-pic}.

Analogous definitions exist for tensors \cite{kolda_tensor_2025}.
A \emph{rank-1 tensor} is the outer product of three vectors, which we denote
as $\va \circ \vb \circ \vc$ and whose $(i,j,k)$ entry is $a_i b_j c_k$.
The \emph{rank} of a tensor $\T \in \Real^{m \times n \times p}$ is 
the minimum $r$ such that $\T$ can be
written as the sum of $r$ three-way vector outer products:
\begin{align}\label{eqn:tensor-rank-decomposition}
  \T
  = \llbracket \A, \B, \C \rrbracket
  \equiv \sum_{\ell=1}^r \mbf{a}_{\ell} \circ \mbf{b}_{\ell} \circ \mbf{c}_{\ell}
  .
\end{align}
Here, $\A \in \Real^{m \times r}$, $\B \in \Real^{n \times r}$, and
$\C \in \Real^{p \times r}$ are the \emph{factor matrices} and
$\mbf{a}_{\ell}$, $\mbf{b}_{\ell}$, and $\mbf{c}_{\ell}$ are the
\emph{factors} representing the $\ell$th columns of $\A$, $\B$, and
$\C$, respectively.
See \cref{fig:tensor-decomp-pic} for a visualization of this decomposition.

\begin{figure}[ht] 
  \centering
  \begin{subfigure}{\textwidth}
    \centering
    \includegraphics{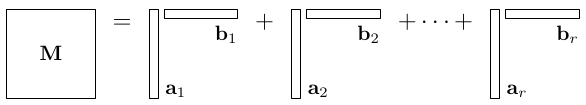}
    \caption{Rank-$r$ matrix decomposition: $\Mx{M} = \sum_{\ell=1}^r \mbf{a}_{\ell} \circ \mbf{b}_{\ell}$.}
    \label{fig:matrix-decomp-pic}
  \end{subfigure}
  \begin{subfigure}{\textwidth}
    \centering 
    \includegraphics{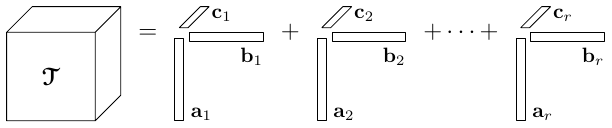}
    \caption{Rank-$r$ tensor decomposition: $\Tn{T} = \sum_{\ell=1}^r \mbf{a}_{\ell} \circ \mbf{b}_{\ell} \circ \mbf{c}_{\ell}$.}
    \label{fig:tensor-decomp-pic}
  \end{subfigure}
  \caption{Matrix and tensor rank decompositions.}
  \label{fig:decomp-pics}
\end{figure}

In contrast to matrix rank,
there is no efficient algorithm for computing tensor rank,
which is known to be NP-hard in general \cite{hillar_most_2013}.
However, the rank of any $\ttt$ tensor can be determined \cite{kruskal_rank_1983,kruskal_rank_1989,ten_berge_kruskals_1991,de_silva_tensor_2008} as we discuss in \cref{sec:hyperd-rank},
and so we focus on $\ttt$ tensors to give a glimpse of the geometric
and algebraic structure of tensors.
We compare the properties of matrix and tensor rank, and delve into why
tensors have the properties that they do.

We call a rank \emph{typical} if it occurs on a set of positive (Lebesgue) measure.
One interesting property of tensors is that there is not necessarily a single typical rank for a tensor of a given size.
For example, the typical rank of a $2 \times 2$ matrix is 2; in other words,
every $2 \times 2$ matrix has rank 2 except for a set of measure zero.
In contrast, the typical rank of a $\ttt$ tensor is 2 or 3; that is, both ranks 2 and 3 occur on sets of positive measure.
(Additionally, the reader may have observed that we are implying the tensor rank can be larger than the maximum dimension of the tensor.
For tensors of size $\ttt$, the maximum possible rank is 3.)

A special polynomial in tensor entries, called the \emph{hyperdeterminant} and discussed in detail in \cref{sec:hyperd-rank}, can be used to determine
if a $\ttt$ tensor has one of the two typical ranks. A tensor is rank-2 if its 
hyperdeterminant is positive and rank-3 if its hyperdeterminant is negative.
The hyperdeterminant is zero on a 7-dimensional hypersurface which 
contains $\ttt$ tensors of all possible ranks (0, 1, 2, and 3).
A three-dimensional slice of the surface is visualized in \cref{fig:affine-intro} using \webtool{/hdet} which contains a button to generate this figure.

\begin{figure}[ht]
  \centering
  \includegraphics[width=0.7\textwidth]{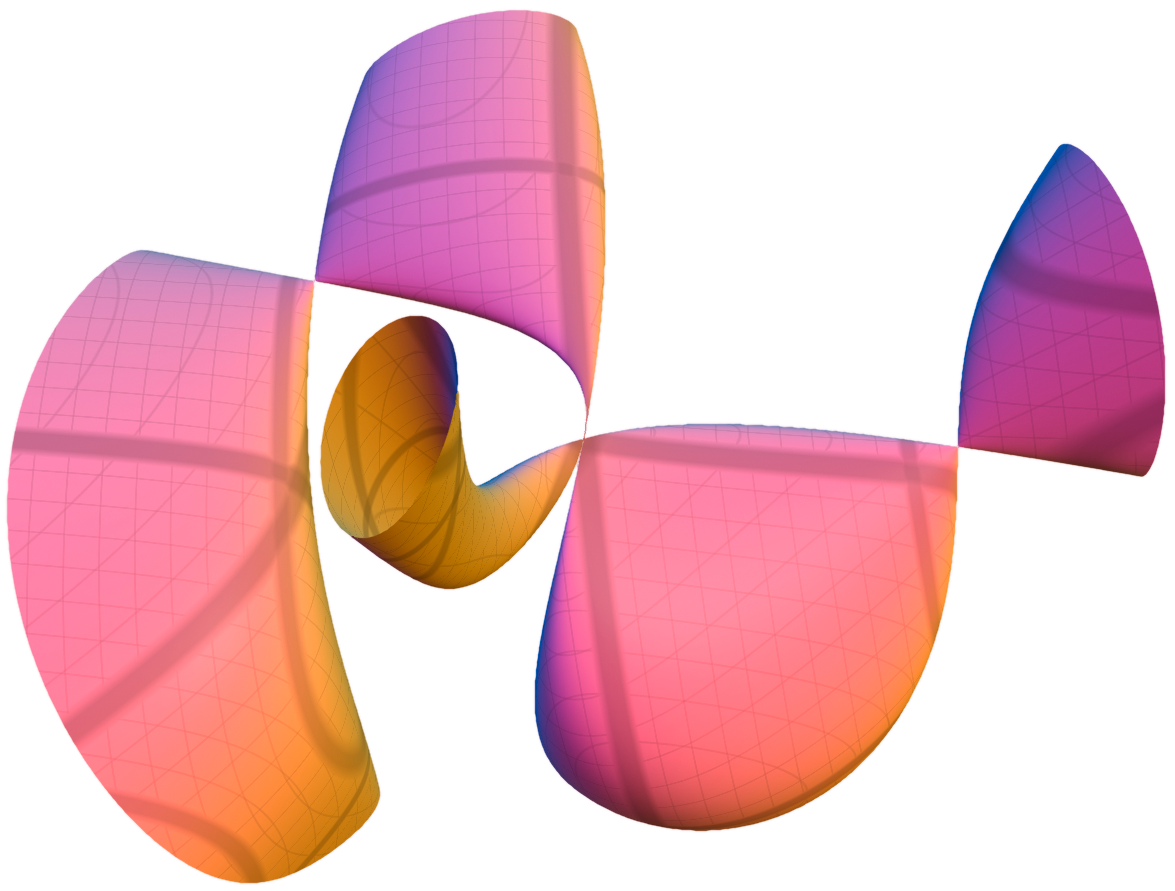}
  \caption{A spherically-clipped 3-dimensional projection of the 8-dimensional space of $\ttt$ tensors showing the 7-dimensional zero set of the hyperdeterminant.
  }
  \label{fig:affine-intro}
\end{figure}

A matrix or tensor is said to have a given \emph{canonical form} if it can 
be transformed into that form by an invertible transformation; 
see \cref{sec:orbits-canon}.
There are eight  canonical forms for $\ttt$ tensors, even though there are only four 
possible ranks (0, 1, 2, and 3). 
In contrast, the number of canonical forms of a matrix is equal to 
the number of possible ranks.

An \emph{orbit} of a canonical form is a set of matrices or tensors 
that are equivalent under invertible transformations.
The simplicity of the matrix case is evident from the topology and 
tesselation of its orbits in the ambient space.
The set of rank-2 matrices of size $2 \times 2$ is a single orbit of 
dimension 4. 
Its closure contains the set of rank-1 matrices, which is a single 
orbit of dimension 3.
The closure of rank-1 matrices contains the rank-0 (all zero entries) matrix.
This leads to a simple hierarchical structure on the orbits.
The situation for $\ttt$ tensors is markedly different.
There are two distinct orbits of sets of rank-2 and rank-3 tensors, each of which have dimension 8. 
There is another orbit of rank-3 tensors of dimension 7 and multiple distinct
orbits of rank-2 tensors, each of dimension 5. 
The closure of rank-3 tensors (the maximum possible rank) does not 
contains all tensors of lower rank, so the orbits are no longer 
hierarchical.
See \cref{sec:alg-geo-struct} for dimensions of matrix and tensor orbits, and \cref{sec:tesselation-topology} for visualizations.

These results on algebraic and geometric structure of $\ttt$ tensors give insight into ill-posed instances of low-rank tensor approximation.
Unlike matrices, for which the best rank-$r$ approximation is well-defined
according to the Eckart-Young theorem \cite{eckart_approximation_1936},
no rank-3 tensor of size $\ttt$ has a best rank-2 approximation \cite{de_silva_tensor_2008}.
In \cref{sec:ill-posedness}, we provide a geometric proof for this 
surprising property and provide intuitive explanations with 
the aid of singular geometry and computational visualizations.
The singularities (e.g., self-intersections) of the zero hyperdeterminant surface in \cref{fig:affine-intro} play an 
important role in the ill-posedness of low-rank approximation.

In \cref{sec:beyond-ttt}, we overview how the
algebraic and geometric structure of tensors changes when we go from 
$\ttt$ tensors to larger scales.
At the same time, we explain how several of the qualitative phenomena observed in this article persist in higher dimensions.
Open research problems for larger-sized tensors are also mentioned.  

\section{(Hyper)determinant and rank}
\label{sec:hyperd-rank}

The algebraic notion of determinant characterizes matrix rank, and an analogous concept plays a role in the characterization of tensor rank for $\ttt$ tensors.
Before we discuss this, we recall some definitions.

A \emph{semi-algebraic} subset of $\Real^n$ is a set of all points satisfying a finite Boolean combination of polynomial equalities and inequalities \cite[Definition 2.14]{bochnak_real_1998}.
This definition may seem limited, but there are a variety of ways to build up complicated semi-algebraic sets.
Indeed, finite combinations of semi-algebraic sets via set unions, intersections, complements, set differences, and Cartesian products also result in semi-algebraic sets.
Furthermore, the image of a semi-algebraic set under any \emph{polynomial map} (that is, a function whose output entries are polynomial functions of the input entries) is again a semi-algebraic set \cite[Proposition 2.2.7]{bochnak_real_1998}, which is a consequence of the Tarski-Seidenberg principle \cite[Section 1.4]{bochnak_real_1998}.

Informally, the \emph{local dimension} of the set $S \subset \Real^n$ at a point in $S$ refers to number of continuous degrees of freedom in $S$ around the point.
The \emph{dimension} of the set $S$ is then simply the maximum over all local dimensions.
For instance, a point 
has dimension zero, a line has dimension 1, a plane
has dimension 2, the surface of the unit sphere in $\mathbb{R}^3$ has dimension 2, the set of $m \times n$ matrices has dimension $mn$, and the set of $\ttt$ tensors has dimension 8.
Formally, since we will be working exclusively with sets defined by polynomials, the dimension refers to the \emph{semi-algebraic dimension} (defined precisely in \cite[Section 2.8]{bochnak_real_1998}).

A set $S \subset \Real^n$ is called (Lebesgue) \emph{measure-zero} if it can be covered by a countable union of $n$-dimensional cubes whose total volume can be made arbitrarily small; for a complete definition see, e.g., \cite{rudin_principles_1976}.
For example, a point is a set of measure zero with respect to a line, a curve is a set of measure zero with respect to a plane, and any set of dimension $\leq r-1$ is measure zero with respect to a set of dimension $r$.
Additionally, any countable union of measure zero sets is measure zero \cite{rudin_principles_1976}, including finite unions.
A set that is not measure zero has \emph{positive volume} or \emph{positive measure}.
We say a property of a point in $\mathbb{R}^n$ is \emph{generic} if it holds everywhere except for a set of measure zero.
For example, being nonzero is a generic property in $\Real$.
If a property holds on a set of positive measure, then we say it is \emph{typical}.
For example, being positive is a typical property in $\Real$.
Any generic property is typical, by definition.

\subsection{Matrix determinant}
The \emph{determinant} of a square matrix is a polynomial function of the matrix entries.
The determinant for the $2 \times 2$ matrix
$\begin{bsmallmatrix}
    a & b \\
    c & d
  \end{bsmallmatrix}$
is $ad - cb$.
A square matrix is non-singular (and hence full rank) if and only if its determinant is nonzero
\cite{horn_matrix_2017}.
More generally, an $m \times n$ matrix $\A$ is rank $r$ if and only if
the determinants of all square submatrices of size $(r+1) \times (r+1)$
are zero and
there exists a square submatrix of size $r \times r$ whose determinant is nonzero \cite{horn_matrix_2017}.

The determinant connects invertibility from linear algebra
to algebraic equations.  
It provides an algebraic way to show that full-rank matrices
are generic, as we now explain:
Observe that the set of $2 \times 2$ matrices
of rank 2 is a semi-algebraic set because it is the set
$\set{ \begin{bsmallmatrix} a & b \\ c & d \end{bsmallmatrix} : ad - cb \neq 0}$.
At any matrix of rank 2, the local dimension is 4, so the set has dimension 4 and positive volume.
This is because the determinant is a continuous function, so at every rank-2 matrix there exists an open ball where the determinant is nonzero.
On the other hand, the set of $2 \times 2$ matrices
of rank $\leq 1$ is the set
$\set{ \begin{bsmallmatrix} a & b \\ c & d \end{bsmallmatrix} : ad - cb = 0}.$
The dimension of this set is 3 since it must be the case that
$a = cb/d$ for generic $d \neq 0$.
Hence one of the coordinates is constrained, and there are only three continuous degrees of freedom.
Therefore the set of $2 \times 2$ matrices of rank $\leq 1$ 
is dimension 3 and measure zero; consequently, the set of $2 \times 2$ matrices of rank 2 is generic. 
More generally, since any rank-deficient matrix in $\Real^{m \times n}$ has a submatrix with determinant zero and 
the zero set of any nonzero polynomial on $\mathbb{R}^{m \times n}$ has dimension at most $mn-1$ {\cite[Section~2.8]{bochnak_real_1998}}, it holds that the set of low-rank matrices has dimension at most $mn - 1$.
Thus the property of being full rank is generic.

\subsection{Cayley's hyperdeterminant}
A polynomial characterization of rank is also possible for $\ttt$ tensors.
Consider the $\ttt$ tensor
\begin{equation}\label{eqn:T}
  \T
  =
  \begin{bmatrix}[c | c]
    \T(:,:,1) & \T(:,:,2)
  \end{bmatrix}
  =
  \begin{bmatrix}[c | c]
    \mbf{T}_1 & \mbf{T}_2
  \end{bmatrix}
  =
  \begin{tensoreight}
    t_{111} & t_{121} & t_{112} & t_{122} \\
    t_{211} & t_{221} & t_{212} & t_{222}
  \end{tensoreight}
  .
\end{equation}
Here, $\mbf{T}_1 \equiv \T(:,:,1)$ and $\mbf{T}_2 \equiv \T(:,:,2)$  are
called the \emph{frontal slices} of the tensor $\T$.
The \emph{hyperdeterminant}%
\footnote{
  In 1983, Kruskal surmised that a fourth degree polynomial can be used
to determine the rank of a $\ttt$ tensor \cite{kruskal_rank_1983,kruskal_rank_1989}.
In 1991, Ten Berge described Kruskal's polynomial as the discriminant of
$\det(\mbf{T}_{2} - \lambda \mbf{T}_{1})$ \cite{ten_berge_kruskals_1991}.
Independently in 1992, Gelfand, Kapranov, and Zelevinsky \cite{gelfand_hyperdeterminants_1992} published
work in algebraic geometry naming this Cayley's hyperdeterminant from Cayley's 1845 work \cite{cayley_theory_1845,cayley_collected_2009}.
} of $\T$ is
\begin{align}
  \label{eqn:hdet}
  \begin{split}
    \Delta(\T)
    = & \left[\frac{\det(\mbf{T}_1 + \mbf{T}_2) - \det(\mbf{T}_1 - \mbf{T}_2)}{2} \right]^2
    - 4 \det(\mbf{T}_1) \det(\mbf{T}_2)                                                     \\
    = & (t_{111}t_{222} + t_{112}t_{221} - t_{121}t_{212} - t_{122}t_{211})^2               \\
      & -4(t_{111}t_{221} - t_{121}t_{211})(t_{112}t_{222} - t_{122}t_{212}).
  \end{split}
\end{align}

The relationship between the hyperdeterminant with the geometry of sets of fixed tensor rank is studied carefully in \cref{sec:orbits-tensor-rank}, but for now we employ the following characterization \cite{ten_berge_kruskals_1991}:
\begin{equation}\label{eqn:hyperdet-rank-simple}
  \rank(\T) =
  \begin{cases}
    2 & \text{if } \Delta(\T)>0, \\
    3 & \text{if } \Delta(\T)<0.
  \end{cases}
\end{equation}
The case when $\Delta(\T)=0$ will be addressed shortly, see \cref{sec:hdet-zero}.

\subsubsection{Rank-2 and rank-3 tensors are typical}
\label{sec:rank-2-3-typical}
From \cref{eqn:hyperdet-rank-simple}, we can make a conclusion quite different from the situation for fixed size matrices:
there is no generic rank; instead, both ranks 2 and 3 are typical for $\ttt$ tensors.
In other words, the set of rank-2 tensors and the set of rank-3 tensors both have positive volume in $\Real^{\ttt}$, as we argue below.

We begin by showing that the set of rank-2 tensors is positive measure.
Observe that
\begin{equation}\label{eqn:rank-2-example}
  \Delta\left(
  \begin{tensoreight}[3ex]
    4 & 2 & 1 & 1 \\
    4 & 3 & 1 & 1
  \end{tensoreight}
  \right) = 1,
\end{equation}
and so this tensor is rank 2.
Since the hyperdeterminant is a continuous function of the tensor entries, there exists an open ball in $\Real^{\ttt}$ centered about this tensor
such that every point in the ball has strictly positive hyperdeterminant.
Such an open ball has positive measure, concluding the argument.
The argument for rank-3 tensors is analogous, based on the observation that
\begin{equation}\label{eqn:rank-3-example}
  \Delta\left(
  \begin{tensoreight}[3ex]
      1 & 0 &  0 & 1 \\
      0 & 1 & -1 & 0
    \end{tensoreight}
  \right) = -4,
\end{equation}
and so this tensor is rank-3 by the characterization \eqref{eqn:hyperdet-rank-simple}.

\subsubsection{Hyperdeterminant zero tensors}
\label{sec:hdet-zero}
Our characterization of tensors using the hyperdeterminant is incomplete because we have not discussed $\Delta(\T)=0$.
As it turns out, the hyperdeterminant alone is insufficient to determine the rank of the tensor:
if $\Delta(\T)=0$ the rank of $\T$ can be 0, 1, 2 or 3.
For example, let us consider how to determine the rank of the following tensor with  hyperdeterminant zero:
\begin{equation}\label{rank-1-example}
  \Delta\left(
  \begin{tensoreight}[3ex]
      1 & 2 & 3 & 6 \\
      1 & 2 & 3 & 6
    \end{tensoreight}
  \right) = 0.
\end{equation}

For this we need an additional tool.
The \emph{multilinear rank} \cite{kolda_tensor_2025}
of a tensor is the tuple of the ranks of its unfoldings:
\begin{displaymath}
  \mrank(\T) = \left( \rank(\TM{1}), \rank(\TM{2}), \rank(\TM{3}) \right) .
\end{displaymath}
The \emph{mode-$k$ unfoldings} \cite{kolda_tensor_2025} 
of a tensor $\T \in \Real^{\ttt}$ are denoted by $\TM{k}$ and given by
\begin{align}\label{eqn:unfoldings}
  \begin{split}
    \TM{1} & =
    \begin{bmatrix}
      t_{111} & t_{121} & t_{112} & t_{122} \\
      t_{211} & t_{221} & t_{212} & t_{222}
    \end{bmatrix}, \\
    \TM{2} & =
    \begin{bmatrix}
      t_{111} & t_{211} & t_{112} & t_{212} \\
      t_{121} & t_{221} & t_{122} & t_{222}
    \end{bmatrix}, \quad \\ %
    \TM{3} & =
    \begin{bmatrix}
      t_{111} & t_{211} & t_{121} & t_{221} \\
      t_{112} & t_{212} & t_{122} & t_{222}
    \end{bmatrix}. \\
  \end{split}
\end{align}

\begin{exercise}[Multilinear rank lower bounds tensor rank]
  \label{exr:mrank-rank}
  Show that any tensor $\T \in \Real^{\ttt}$ satisfies
 $\max \left( \rank(\TM{1}), \rank(\TM{2}), \rank(\TM{3}) \right) \leq \rank(\T)$.
 In other words, show that the maximum flattening rank is a lower bound for the rank.
 \textit{Hint:} Use the fact that $\T = \sum_{i=1}^r \mathbf{a}_i \circ \mathbf{b}_i \circ \mathbf{c}_i$ for some $\mathbf{a}_i, \mathbf{b}_i, \mathbf{c}_i$, where $r = \rank(\T)$.
\end{exercise}

With both the hyperdeterminant and multilinear rank at our disposal,
we can compute the rank of any tensor in $\Real^{\ttt}$ \cite{kruskal_rank_1983,kruskal_rank_1989,ten_berge_kruskals_1991,de_silva_tensor_2008}.
\Cref{eqn:hyperdet-rank-simple} specifies the rank if $\Delta(\T) \neq 0$.
If $\Delta(\T)=0$, then
\begin{equation}\label{eqn:rank-determination}
  \rank(\T) =
  \begin{cases}
    0 & \text{if } \mrank(\T) = (0,0,0), \\
    1 & \text{if } \mrank(\T) = (1,1,1), \\
    3 & \text{if } \mrank(\T) = (2,2,2), \\
    2 & \text{otherwise}.                \\
  \end{cases}
\end{equation}
In the last case, $\mrank(\T)$ is some permutation of $(2,2,1)$.
All realized combinations of hyperdeterminant, multilinear rank, and rank are cataloged in  \cref{tab:tensor-orbits};
the other columns of \cref{tab:tensor-orbits} will be explained in subsequent sections.
 This table is adapted from de Silva and Lim \cite{de_silva_tensor_2008}, building on Gelfand, Kapranov, and Zelevinsky \cite{gelfand_hyperdeterminants_1992}.
We remark that the definition of the hyperdeterminant from \cite{gelfand_hyperdeterminants_1992} (of which \cref{eqn:hdet} is a special case for $\ttt$ tensors) applies to a variety of tensor shapes beyond $\ttt$.  However, its complexity grows drastically with tensor size and its role in the characterization of rank is less clear at present \cite{de_silva_tensor_2008}.

\begin{exercise}
  Using \cref{eqn:rank-determination},
 determine the rank of the tensor in \cref{rank-1-example}.
\end{exercise}

\begin{exercise}[\textit{Challenge}]
  Without using \cref{eqn:rank-determination},
  prove that if two unfoldings of a $\ttt$ tensor have rank 1,
  then the tensor is rank 1.
\end{exercise}

\subsection{Rank-constrained sets are semi-algebraic}
\label{sec:rank-semialgebraic}
We will make heavy use of sets of matrices or tensors of fixed or bounded rank, e.g., the set of $\ttt$ tensors of rank $\leq 2$.
A fundamental but simple result is that all such sets are semi-algebraic (described by a finite number of polynomial equalities and inequalities), which can be shown using the minimal amount of technology for semi-algebraic sets introduced at the start of \cref{sec:hyperd-rank}.

We previously discussed how matrices are semi-algebraic using the determinant. Another way to
see that matrices of bounded rank are semi-algebraic is by the definition of matrix rank: $\rank(\A \B^{\Tr}) \leq r$ if $\A$ and $\B$ have $r$ columns.
Since the set
$\Real^{m \times r} \times \Real^{n \times r}$
is semi-algebraic and that the map
$(\A, \B) \mapsto \A \B^{\Tr}$
is polynomial, we have that the image (all rank $\leq r$ matrices) is a semi-algebraic set \cite[Proposition 2.2.7]{bochnak_real_1998}.
Sets of fixed matrix rank can be realized as set differences of sets with bounded matrix rank, and hence are also semi-algebraic.
The proof for tensors uses an analogous factorization to parameterize bounded rank tensors, but is otherwise identical \cite[Theorem 6.2]{de_silva_tensor_2008}.

\begin{exercise}
    Show that the set of $m \times n \times p$ tensors of rank $\leq r$ is semi-algebraic.
\end{exercise}

The significance of these statements is that, since we can check membership in a semi-algebraic set by evaluating a finite number of polynomials, we can determine the rank of a matrix or tensor by evaluating a finite number of polynomials \cite[Theorem 6.3]{de_silva_tensor_2008}.
For matrix rank, this is achieved by evaluating the minors.
For tensor rank, however, we already know that the hyperdeterminant is insufficient for $\ttt$ tensors, and a full semi-algebraic description also involves the multilinear rank.
In general, a concrete semi-algebraic description in terms of polynomials is not known for $m \times n \times p$ tensors beyond rank-$2$ \cite{seigal_real_2017}.

\begin{table}
  \begin{center}
    \begin{tblr}{c | c | c | c | c | c | c | c}
      orbit                & canonical                             & hyper- & multi-    & rank & border & dimen- & proba-              \\[-1ex]
      & form                                  & deter- & linear    &      & rank   & sion   & bility              \\[-1ex]
      &                                       & minant & rank      &      &        &        &                     \\ \hline
      $G_3$                & \tabletensor{1}{0}{0}{-1}{0}{1}{1}{0} & $-$    & $(2,2,2)$ & 3    & 3      & 8      & $1 - \frac{\pi}{4}$ \\ \hline
      $G_2$                & \tabletensor{1}{0}{0}{0}{0}{0}{0}{1}  & $+$    & $(2,2,2)$ & 2    & 2      & 8      & $\frac{\pi}{4}$     \\ \hline
      $D_3$                & \tabletensor{1}{0}{0}{1}{0}{0}{1}{0}  & 0      & $(2,2,2)$ & 3    & 2      & 7      & 0                   \\ \hline
      $D_2$                & \tabletensor{1}{0}{0}{0}{0}{1}{0}{0}  & 0      & $(2,2,1)$ & 2    & 2      & 5      & 0                   \\ \hline
      $D_2^\prime$         & \tabletensor{1}{0}{0}{1}{0}{0}{0}{0}  & 0      & $(1,2,2)$ & 2    & 2      & 5      & 0                   \\ \hline
      $D_2^{\prime\prime}$ & \tabletensor{1}{0}{0}{0}{0}{0}{1}{0}  & 0      & $(2,1,2)$ & 2    & 2      & 5      & 0                   \\ \hline
      $D_1$                & \tabletensor{1}{0}{0}{0}{0}{0}{0}{0}  & 0      & $(1,1,1)$ & 1    & 1      & 4      & 0                   \\ \hline
      $D_0$                & \tabletensor{0}{0}{0}{0}{0}{0}{0}{0}  & 0      & $(0,0,0)$ & 0    & 0      & 0      & 0
    \end{tblr}
    \caption{
      Orbits of $\ttt$ real tensors and associated quantities and invariants.
      $D$ orbits are ``degenerate'' (non-typical) while $G$ orbits are ``generic'' (typical).
      Adapted from \cite{kruskal_rank_1983,de_silva_tensor_2008}.
      The results in the probability column are due to Bergqvist and assume tensor entries are standard normal i.i.d.~\cite{bergqvist_exact_2013}.
      Understanding and justifying this table will be a focus of upcoming sections.
    }
    \label{tab:tensor-orbits}
  \end{center}
\end{table}

\section{Orbits and canonical forms}
\label{sec:orbits-canon}
In this section we explain how to use invertible transformations, or more precisely an appropriate group action, to partition the space of matrices or tensors into subsets of constant rank.

\subsection{Groups and group actions}
We begin by recalling the mathematical concepts of groups, group actions, and their orbits.

A \emph{group} $G$ is a set along with a binary operation
$\ast: G \times G \rightarrow G$ such that the following properties hold:
\begin{itemize}
  \item associativity: $a \ast b \ast c = (a \ast b ) \ast c = a \ast (b \ast c)$ for all $a,b,c \in G$,
  \item identity: there exists $e \in G$ such that
        $e \ast a = a \ast e = a$ for all $a \in G$, and
  \item inverse: for all $a \in G$ there exists $b \in G$
        such that $b \ast a = a \ast b = e$.
\end{itemize}
We are interested in groups where the elements are matrices and the operator $\ast$ is matrix multiplication.
Three important matrix groups are the general linear, special linear, and orthogonal groups,
defined respectively as
\begin{align*}
    \GL(n) &= \set{ \Mx{M} \in \Real^{n \times n} : \det(\Mx{M}) \neq 0 }, \\
    \SL(n) &= \set{ \Mx{M} \in \Real^{n \times n} : \det(\Mx{M}) = 1}, \\ %
    \O(n) &= \set{ \mbf{M} \in \Real^{n \times n} :
      \mbf{M}^{\Tr} \mbf{M} = \mbf{M} \Mx{M}' = \mbf{I}_{n} }.
\end{align*}

\begin{exercise}
  Prove that $\GL$, $\SL$, and $\O$ are groups.
\end{exercise}

In addition to the binary operation $\ast$ occurring within the group, we can define an operation of group elements on objects \emph{outside} the group.
If we have a group $G$ and a set $X$, we call a map $\cdot: G \times X \rightarrow X$ a \emph{group action} if for all $x \in X$ and $g, h \in G$ it holds
$g \cdot (h \cdot x) = (g \ast h) \cdot x$ 
and $e \cdot x = x$.
For example, if $G=\O(n)$ and $X = \Real^n$, we can define a group action via matrix-vector multiplication.

Given a group $G$ with operator $\ast$ and action $\cdot$ on $X$,
the \emph{orbit} of $x \in X$ under $G$ is denoted as $G \cdot x$ and
defined to be the set of images of $x$ under  elements of $G$, i.e.,
\begin{displaymath}
  G \cdot x = \set{ g \cdot x : g \in G } \subset X.
\end{displaymath}
Each point in $X$ gives rise to an orbit under a group action. Furthermore, it is straightforward to prove that two orbits are either disjoint or equal.
Therefore, the action of a group $G$ induces a partition of $X$ into orbits.
To make the concept of an orbit concrete we give a simple example.

\begin{example}
  Consider the vector
  $\mbf{x} = \begin{bsmallmatrix} 1 \\ 0\end{bsmallmatrix} \in \Real^2$
 and the action of $\O(2)$ on $\mathbb{R}^2$ via matrix-vector multiplication.
  Then the orbit is $\O(2) \cdot \mbf{x} = \set{\mbf{y} \in \Real^2 | \norm{\mbf{y}}_2 =1}$ as shown in \cref{fig:orbit}.  Following the same logic, other orbits under the group action are circles of other radii centered at the origin, as well as the singleton $\{\begin{bsmallmatrix} 0 \\ 0\end{bsmallmatrix} \}$.
\end{example}

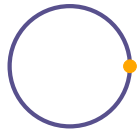
\begin{figure}[ht]
  \centering
  \begin{tikzpicture}[scale=0.8]
    \draw[ultra thick, color=dark] (1,0) arc (0:360:1);
    \filldraw[light] (1,0) circle (3pt);
  \end{tikzpicture}
  \caption{Orbit (blue circle) of
    $\mbf{x} = \begin{bsmallmatrix} 1 \\ 0\end{bsmallmatrix} \in \Real^2$
    (orange dot) under the action of $\O(2)$}
  \label{fig:orbit}
\end{figure}

Given groups $G_1$ and $G_2$ with respective operators $\ast_1$ and $\ast_2$, one can define the \emph{product group} $G = G_1 \times G_2$. 
As a set, the product group is the Cartesian product.
The group operation $\ast$ for the product group $G$ is simply the operator $\ast_1$ in the first argument and $\ast_2$ in the second:
$(g_1,g_2) \ast (h_1,h_2) = (g_1 \ast_1 h_1, g_2 \ast_2 h_2)$.
A product group is itself a group.
For example, the orbit of $\A \in \mathbb{R}^{m \times n}$ under $\GL(m) \times \GL(n)$ is
\begin{displaymath}
  \left(\GL(m) \times \GL(n)\right) \cdot \A \equiv
  \set{\U \A \V^{\Tr}: \U \in \GL(m), \V \in \GL(n)},
\end{displaymath}
where the group action is defined to be two-sided matrix multiplication with a transpose on the right side.

\begin{exercise}[\textit{Challenge}]
  Consider the space $\Real^{m \times n}$ under the group action of
  $\O(m) \times \O(n)$.
  Into how many sets is $\Real^{m \times n}$ partitioned by this group action?  Describe the orbits.
  \emph{Hint}: Consider the singular value decomposition (SVD) of matrices \cite{GoVa96}.
\end{exercise}

\subsection{Orbits and matrix rank}
\label{sec:orbit-matrix}
Now we examine the role of orbits in matrix rank.
Without loss of generality (since rank is unaffected by transposition),
$m \geq n$.

\begin{proposition}[Matrix orbits]
  \label{prp:equiv}
  If $\Mx{M} \in \Real^{m \times n}$ has rank $r$ and $m \geq n$, then %
  \begin{displaymath}
      \left(\GL(m) \times \GL(n)\right) \cdot \Mx{M} =
      \left(\GL(m) \times \GL(n)\right) \cdot \C_r
      \quad\text{where}\quad
      \C_r \equiv
      \begin{bsmallmatrix}
          \Mx{I}{r} & \Mx{0} \\
          \Mx{0} & \Mx{0}
      \end{bsmallmatrix} \in \Real^{m \times n}.
  \end{displaymath}
\end{proposition}
Noting that $(\A,\B) \in \GL(m) \times \GL(n)$ implies $\rank(\A \C_r \B^{\Tr}) = r$, the set of rank-$r$ matrices is precisely the orbit of the partial identity matrix $\C_r$ under the action of $\GL(m) \times \GL(n)$.

We call the matrix $\C_r$ in \cref{prp:equiv} a \emph{canonical form}, as it
represents a simple rank-$r$ matrix and any other rank-$r$ matrix can be transformed into $\C_r$.
This particular choice of orbit representative is simply a convention, and any other rank-$r$ matrix would also work.
Since every matrix is in the orbit with canonical form $\Mx{C}{r}$ for some $r \in \set{0,...,n}$, the space $\Real^{m \times n}$ is partitioned into $n+1$ orbits.

\begin{exercise}
  Prove \cref{prp:equiv}.
  \emph{Hint}: Use the SVD to show that any rank-$r$ matrix has a factorization $\Mx{X} \C_r \Mx{Y}^{\Tr}$ where
  $\Mx{X} \in \GL(m), \Mx{Y} \in \GL(n)$.
\end{exercise}
\begin{exercise}
    By \cref{prp:equiv} any $m \times n$ matrix of rank $r$ can be written as
    $\Mx{X}\C_r \Mx{Y}^{\Tr}$.
    Explain why the matrices $\Mx{X}$ and $\Mx{Y}$ are not unique.
\end{exercise}

\subsection{Orbits and tensor rank}\label{sec:orbits-tensor-rank}
Here we consider the orbits of real $\ttt$ tensors under the group action of
$\GL(2) \times \GL(2) \times \GL(2)$.
We will demonstrate that this partitions the space of $\ttt$ tensors into a finite number of orbits.

\subsubsection{Tensor times matrix (TTM)}
The \emph{tensor-times-matrix (TTM)} product,
denoted $\T \times_k \U$
transforms a tensor $\T$ by multiplying it in mode-$k$ by the
matrix $\U$. If we say
$\Tn{S} = \T \times_k \U$, this is equivalent to saying that their
mode-$k$ unfoldings, defined in \cref{eqn:unfoldings},  are related as
\begin{displaymath}
  \Mx{S}{(k)} = \U\TM{k}.
\end{displaymath}
The TTM has the following properties:
\begin{gather}
  \label{eqn:mode-k-double}
  (\T \times_k \Mx{X}) \times_k \Mx{Y} = \T \times_k \Mx{YX}, \\
  \label{eqn:associative}
  (\T \times_k \Mx{X}) \times_{\ell} \Mx{Y} 
  = (\T \times_{\ell} \Mx{Y}) \times_k \Mx{X} =  \T \times_k \Mx{X} \times_{\ell} \Mx{Y}
  \quad\text{if}\quad k \neq \ell.
\end{gather}
A TTM product computed in multiple modes is the \emph{multi-TTM} operation, also known as the Tucker product.
For a tensor
$\T \in \Real^{m \times n \times p}$ and matrices
$\U \in \Real^{q \times m}$,
$\V \in \Real^{r \times n}$, and
$\W \in \Real^{s \times p}$,
the multi-TTM operation is written as
\begin{displaymath}
  \Tn{S} = \T \times_1 \U \times_2 \V \times_3 \W.
\end{displaymath}
The resulting tensor $\Tn{S}$ is of size $q \times r \times s$ and its
$(\alpha,\beta,\gamma)$-entry is defined explicitly by
\begin{displaymath}
  s_{\alpha \beta \gamma} = \sum_{i=1}^m \sum_{j=1}^n \sum_{k=1}^p
  t_{ijk} u_{\alpha i} v_{\beta j} w_{\gamma k} .
\end{displaymath}
The multi-TTM gives a group action of $\GL(2) \times \GL(2) \times \GL(2)$ on $\ttt$ tensors.

\begin{example}[TTM]
  \label{exa:reduction}
  Consider the tensor from \cref{eqn:rank-2-example}:
  \begin{align}
    \T =
    \begin{tensoreight}
      4 & 2 & 1 & 1 \\
      4 & 3 & 1 & 1
    \end{tensoreight} .
  \end{align}
  Single-mode TTM examples are:
  \begin{align}
    \begin{split}
      \T \times_3
      \begin{bmatrix}[r r]
         1 & -4 \\
         0 &  1
      \end{bmatrix}
      &=
      \begin{tensoreight}
         0 & -2 & 1 & 1 \\
         0 & -1 & 1 & 1
      \end{tensoreight}
      \equiv \Tn{U}, \\
      \Tn{U} \times_1
      \begin{bmatrix}[r r]
         1 & -1 \\
         0 &  1
      \end{bmatrix}
      &=
      \begin{tensoreight}
        0 & -1 & 0 & 0 \\
        0 & -1 & 1 & 1
      \end{tensoreight}
        \equiv \Tn{V}, \\
      \Tn{V} \times_1
      \begin{bmatrix}[r r]
         -1 & 0 \\
         -1 & 1
      \end{bmatrix}
      &=
      \begin{tensoreight}
        0 & 1 & 0 & 0 \\
        0 & 0 & 1 & 1
      \end{tensoreight}
      \equiv \Tn{W}, \\
      \Tn{W} \times_2
      \begin{bmatrix}[r r]
          1 & 0 \\
         -1 & 1
      \end{bmatrix}
      &=
      \begin{tensoreight}
        0 & 1 & 0 & 0 \\
        0 & 0 & 1 & 0
      \end{tensoreight}\!\!.
    \end{split}
  \end{align}
  Combining all four operations into a multi-TTM according to \cref{eqn:associative}, one finds the same result:
  \begin{align*}
      \T \times_1
      \begin{bmatrix}[r r]
        -1 & 1 \\
        -1 & 2
      \end{bmatrix}
      \times_2
      \begin{bmatrix}[r r]
         1 & 0 \\
        -1 & 1
      \end{bmatrix}
      \times_3
      \begin{bmatrix}[r r]
         1 & -4 \\
         0 &  1
      \end{bmatrix}
      &=
      \begin{tensoreight}
         0 & 1 & 0 & 0 \\
         0 & 0 & 1 & 0
      \end{tensoreight}
  \end{align*}
  where the two mode-1 matrices have been combined into the appropriate product.
\end{example}

\subsubsection{Group action on tensors and invariants} %
The group action of  $\GL(2) \times \GL(2) \times \GL(2)$ 
on $\ttt$ tensors
is natural in the sense that important properties are invariant under the action,
including rank,
multilinear rank,
sign of the hyperdeterminant, and
border rank (to be discussed in \cref{sec:ill-posedness})
\cite{de_silva_tensor_2008}.
Further, the hyperdeterminant is a \emph{relative invariant}
under of the action of $\GL(2) \times \GL(2) \times \GL(2)$, here meaning
\begin{displaymath}
  \Delta( \T \times_1 \U \times_2 \V \times_3 \W )
  = \Delta(\T) \det(\U)^2 \det(\V)^2 \det(\W)^2 .
\end{displaymath}

\begin{exercise}
    Show that multilinear rank is invariant under the group action of
    $\GL(2) \times \GL(2) \times \GL(2)$ acting by multi-TTM.
\end{exercise}

\subsubsection{Orbits of \texorpdfstring{$\ttt$}{2 x 2 x 2} tensors}
We can now prove that \cref{tab:tensor-orbits} above provides an exhaustive list of orbits for real $\ttt$ tensors.
The following result explains the second column of \cref{tab:tensor-orbits}, which lists the canonical forms defined by the action of $\GL(2) \times \GL(2) \times \GL(2)$.

\begin{theorem}[Tensor orbits {\cite[Theorem 7.1]{de_silva_tensor_2008}}]
  \label{thm:tensor-orbits}
  \cref{tab:tensor-orbits} contains all canonical forms of real $\ttt$ tensors.
  In other words,
  every tensor in $\Real^{\ttt}$ may be brought to one (and only one) of these forms through the action $\GL(2) \times \GL(2) \times \GL(2)$.
\end{theorem}
\begin{proof}
  A full proof proceeds by an exhaustive analysis of cases depending on the ranks of slices of $\T \in \Real^{\ttt}$.
  In each case below, TTM operations are used to perform reductions on the tensor.  Note that a permutation of slices of the tensor along one of its modes is also a TTM operation.

  \textbf{Case 1: $\rank(\mbf{T}_1) = 0$.}
  In this case, we have 
  \begin{align*}
    \T = \begin{tensoreight}[3ex]
           0 & 0 & \times & \times \\
           0 & 0 & \times & \times
         \end{tensoreight} ,
  \end{align*}
  where we use $\times$ to denote an element whose value is unknown.
  From here, TTM with any
  $(\U,\V,\mbf{I}_2) \in \GL(2) \times \GL(2) \times \GL(2)$
  will preserve the zero matrix in the first frontal slice, while for appropriately chosen $\U$ and $\V$ it will bring the second frontal slice into one of the three following forms depending on $\text{rank}(\mbf{T}_2)$:
  \begin{align*}
    \begin{tensoreight}[3ex]
      0 & 0 & 0 & 0 \\
      0 & 0 & 0 & 0
    \end{tensoreight}
    ,\quad
    \begin{tensoreight}[3ex]
      0 & 0 & 1 & 0 \\
      0 & 0 & 0 & 0
    \end{tensoreight}
    ,\quad \text{or} \quad
    \begin{tensoreight}[3ex]
      0 & 0 & 1 & 0 \\
      0 & 0 & 0 & 1
    \end{tensoreight}.
  \end{align*}
  These three tensors are (after possible mode-3 permutations) equivalent to the canonical forms of $D_0$, $D_1$, and $D_2$, respectively.

  \textbf{Case 2: $\rank(\mbf{T}_1) = 1$.}
  Using elimination matrices and possibly permutations in modes 1 and 2, such a $\T$ can be brought into the form
  \begin{align*}
    \begin{tensoreight}[3ex]
      1 & 0 & a & b \\
      0 & 0 & c & d
    \end{tensoreight}.
  \end{align*}
  
  If $d \neq 0$, then the tensor may be brought to the canonical form of orbit $G_2$, i.e., 
  \begin{align*}
    \begin{tensoreight}[3ex]
      1 & 0 & 0 & 0 \\
      0 & 0 & 0 & 1
    \end{tensoreight}
  \end{align*}
  using mode-1 and 2 elimination matrices to eliminate $b$ and $c$, followed by a mode-3 operation to eliminate the remaining $(1,1,2)$ entry.
  
  If $d = 0$, mode-3 elimination can be used to bring the tensor into the form
  \begin{align*}
    \begin{tensoreight}[3ex]
      1 & 0 & 0 & b \\
      0 & 0 & c & 0
    \end{tensoreight}.
  \end{align*}
  Entries $b$ and $c$ can be independently normalized using mode-2 and mode-1 operations, respectively. Accounting for the fact that $b$ or $c$ could be zero, we obtain the following four possible forms
  \begin{align*}
    \begin{tensoreight}[3ex]
      1 & 0 & 0 & 0 \\
      0 & 0 & 0 & 0
    \end{tensoreight}, \;
    \begin{tensoreight}[3ex]
      1 & 0 & 0 & 1 \\
      0 & 0 & 0 & 0
    \end{tensoreight}, \;
    \begin{tensoreight}[3ex]
      1 & 0 & 0 & 0 \\
      0 & 0 & 1 & 0
    \end{tensoreight}, \;
    \begin{tensoreight}[3ex]
      1 & 0 & 0 & 1 \\
      0 & 0 & 1 & 0
    \end{tensoreight}.
  \end{align*}
  These are the canonical forms of orbits $D_1$, $D_2^\prime$, $D_2^{\prime\prime}$, and $D_3$, respectively.

  \textbf{Case 3: $\rank(\mbf{T}_1) = 2$.}
    Multiplying by $\mbf{T}_1^{-1}$ in the first mode we can achieve the form
  \begin{align*}
    \begin{tensoreight}[3ex]
      1 & 0 & \times & \times \\
      0 & 1 & \times & \times
    \end{tensoreight} .
  \end{align*}
  From here, one can act by some
  $(\U,\U^{-\Tr},\mbf{I}_2) \in \GL(2) \times \GL(2) \times \GL(2)$,
  which will preserve the identity in the first frontal slice, while bringing the second frontal slice into real Jordan canonical form \cite{horn_matrix_2017}.
  Depending on the structure of $\mbf{T}_2$ (repeated eigenvalues and diagonalizable, repeated eigenvalues and non-diagonalizable, distinct real eigenvalues, or complex conjugate eigenvalues)
  the tensor will be one of the four forms
  \begin{align*}
    \begin{tensoreight}[3ex]
      1 & 0 & \lambda &       0 \\
      0 & 1 &       0 & \lambda
    \end{tensoreight} , \;
    \begin{tensoreight}[3ex]
      1 & 0 & \lambda &       1 \\
      0 & 1 &       0 & \lambda
    \end{tensoreight} , \;
    \begin{tensoreight}[3ex]
      1 & 0 & \lambda &   0 \\
      0 & 1 &       0 & \mu
    \end{tensoreight} , \;
    \begin{tensoreight}[3ex]
      1 & 0 &  a & b \\
      0 & 1 & -b & a
    \end{tensoreight},
  \end{align*}
  which after slice reductions may respectively be brought into the following
  \begin{align*}
    \begin{tensoreight}[3ex]
      1 & 0 & 0 & 0 \\
      0 & 1 & 0 & 0
    \end{tensoreight} , \;
    \begin{tensoreight}[3ex]
      1 & 0 & 0 & 1 \\
      0 & 1 & 0 & 0
    \end{tensoreight} , \;
    \begin{tensoreight}[3ex]
      1 & 0 & 0 & 0 \\
      0 & 0 & 0 & 1
    \end{tensoreight} , \;
    \begin{tensoreight}[3ex]
      1 & 0 &  0 & 1 \\
      0 & 1 & -1 & 0
    \end{tensoreight} .
  \end{align*}
  These are the canonical forms of $D_2$, $D_3$ (after permutations in modes 2 and 3), $G_2$, and $G_3$, respectively.

  Accounting for redundancy between the three cases, there are 8 unique canonical forms, as listed in \cref{tab:tensor-orbits}.  We finish by noting that no two of the canonical forms there are in the same orbit, because they have a different sign of  hyperdeterminant or a different  multilinear rank, both of which are  constant over orbits.
\end{proof}

\subsubsection{Conversion to canonical form}
The proof of \cref{thm:tensor-orbits} is constructive and thus gives a procedure to compute three invertible matrices which transform a tensor to its canonical form or vice versa.
That is, given any $\T \in \Real^{\ttt}$ there exists an efficient algorithm to determine the corresponding canonical form $\Tn{G}$ (from \cref{tab:tensor-orbits}) together with matrices $\A, \B, \C \in \GL(2)$ such that
\begin{displaymath}
  \Tn{G}  = \T \times_1 \A \times_2 \B \times_3 \C.
\end{displaymath}
The algorithm proceeds essentially as the proof, but records and accumulates the sequence of matrices used in the multi-TTMs.
For example, suppose that $\T$ has $\rank(\mbf{T}_1) = 1$ and $\T \in D_3$,
i.e., we are in the second case of the proof with $d = 0$.
Then we require in total: an initial transformation of $\Mx{T}{1}$ to $\begin{bsmallmatrix}
    1  & 0 \\ 0 & 0
\end{bsmallmatrix}$ via one mode-1 and mode-2 elimination matrices ($\A_1, \B_2$), a mode-3 elimination matrix ($\C_3$), and two diagonal matrices in modes 1 and 2 ($\A_4, \B_5$) to perform the final scaling.
That is we find matrices such that
\begin{displaymath}
  \G = \T \times_1 (\A_4 \A_1) \times_2 (\B_5 \B_2) \times_3 \C_3.
\end{displaymath}

\begin{exercise}
    For each item below, find matrices $\U,\V,\W \in \GL(2)$ producing the desired result:
    \begin{enumerate}[label={(\roman*)}]
    \item $
    \begin{tensoreight}[2ex]
      0 & 0 & 1 & 0 \\
      0 & 0 & 0 & 0
    \end{tensoreight} \times_1 \Mx{U} \times_2 \Mx{V} \times_3 \Mx{W}
    =
    \begin{tensoreight}[2ex]
      1 & 0 & 0 & 0 \\
      0 & 0 & 0 & 0
    \end{tensoreight}. $
    \item $
    \begin{tensoreight}[2ex]
      1 & 0 & 0 & 1 \\
      0 & 1 & 0 & 0
    \end{tensoreight}
    \times_1 \Mx{U} \times_2 \Mx{V} \times_3 \Mx{W}
    =
    \begin{tensoreight}[2ex]
      1 & 0 & 0 & 1 \\
      0 & 0 & 1 & 0
    \end{tensoreight}. $
    \item $
    \begin{tensoreight}[2ex]
      0 & 1 & 0 & 0 \\
      0 & 0 & 1 & 0
    \end{tensoreight}
    \times_1 \Mx{U} \times_2 \Mx{V} \times_3 \Mx{W}
    =
    \begin{tensoreight}[2ex]
      1 & 0 & 0 & 0 \\
      0 & 0 & 0 & 1
    \end{tensoreight}$.
  \end{enumerate}
  For item (iii), combine $\U, \V, \W$ with the three matrices from \cref{exa:reduction} to transform
    $
    \begin{tensoreight}[2ex]
      4 & 2 & 1 & 1 \\
      4 & 3 & 1 & 1
    \end{tensoreight}$
    to its canonical form.
\end{exercise}

\begin{exercise}
  The matrices $\A, \B, \C$ transforming between an arbitrary tensor and its canonical form are not unique: if $\A, \B, \C$ perform the correct transformation then so do matrices $\alpha \A, \beta \B, \frac{1}{\alpha\beta} \C$.
  For each of the eight orbits, look for additional ambiguities which result in non-uniqueness of the matrices transforming from tensor to canonical form.
\end{exercise}

\begin{example}[Fast algorithm for multiplying complex numbers]
  Consider the product of two complex numbers.
  If $u = a + i b$ and $v = c + i d$ with $a,b,c,d \in \mathbb{R}$ and $i = \sqrt{-1}$, then $z = uv = (ac - bd) + i (ad + bc)$.
  Since this operation is linear in each argument ($u$ and $v$), it may be represented by a bilinear operator
  $\Tn{C} =
   \begin{tensoreight}[2ex]
     1 &  0 & 0 & 1 \\
     0 & -1 & 1 & 0
   \end{tensoreight}$,
  which operates as
  $\mbf{z} = \Tn{C} \times_1 \mbf{u} \times_2 \mbf{v}$.
  Here $\mbf{u}, \mbf{v}, \mbf{z} \in \Real^2$ are representations of $u, v, z$ as real vectors obtained by separating real and imaginary parts
  (e.g., $\mbf{u} = \begin{bsmallmatrix} a \\ b \end{bsmallmatrix}$).

  The \emph{bilinear complexity} of a bilinear operator (informally) measures the required number of multiplications between two non-constants.  It is an important quantity in arithmetic complexity theory.
  Since the bilinear complexity is equivalent to the tensor rank of the corresponding tensor \cite{strassen_vermeidung_1973}, the problem of finding the fastest algorithm to apply a bilinear operator can be solved by finding a decomposition of minimal rank: it's clear if
  $\Tn{C} = \sum_i^r \mbf{a}^{(i)} \circ \mbf{b}^{(i)} \circ \mbf{c}^{(i)}$
  then the contraction with the factorized form
  $\Tn{C} \times_1 \mbf{u} \times_2 \mbf{v} = \sum_i (\mbf{a}^{(i)^{\Tr}} \mbf{u}) (\mbf{b}^{(i)^{\Tr}} \mbf{v}) \mbf{c}^{(i)}$
  gives an algorithm to compute the output of the operator with only $r$ non-constant multiplications.
  Recalling $\mbf{a}^{(i)}, \mbf{b}^{(i)}, \mbf{c}^{(i)}$ are constant, the cost comes from the multiplication of the inner products (one per summand).
  Examining the given $\Tn{C}$ for multiplying complex numbers, it is permutation equivalent to the canonical form of $G_3$, hence it has rank 3 and there exists an algorithm with bilinear complexity 3.
  This is surprising, since the naive algorithm for multiplying complex numbers requires 4 such multiplications.

  This fast bilinear algorithm can be directly applied to accelerate the product of complex matrices  (where multiplication is asymptotically more costly than addition) rather than just the multiplication of complex scalars \cite{higham_stability_1992}.
  The same framework may applied to the bilinear operator corresponding to real matrix multiplication; this can lead one to the famed ``Strassen algorithm'' for multiplying $2 \times 2$ matrices (and for $n \times n$ via recursion) \cite{strassen_gaussian_1969,kolda_tensor_2025}.
  Additional treatment of tensors and arithmetic complexity can be found in \cite{bini_role_2007,kolda_tensor_2025}.
\end{example}

\section{Geometric structure}
\label{sec:alg-geo-struct}
In previous sections, we saw how
product groups of $\GL(2)$ reveal structure on 
$\Real^{2 \times 2}$ and $\Real^{\ttt}$.
For the matrix case, there are three orbits (one per rank);
and for tensors, there are eight orbits (with more than one for certain ranks).
Digging deeper, it is natural to ask the following:
\begin{itemize}
  \item What is the dimension of each orbit?
  \item How do the orbits tesselate to fill the ambient space?
  \item What is the topology (open, closed, or neither) of each orbit?
\end{itemize}
\noindent In this section, we compute the dimensions of the matrix and tensor orbits.
The next section is dedicated to answering the second and third questions.

\subsection{Orbit dimension}
The following result ties the dimension of the image of polynomial maps $f$ to the rank of their Jacobian matrix $Df$.
Recall that the Jacobian matrix has entries equal to the first-order partial derivatives of $f$.  It defines a linear map from the tangent space of the domain to that of the codomain.

\begin{proposition}[consequence of  {\cite[Theorem~5.57]{basu_algorithms_2006}}]
  \label{prp:jacobian-dimension}
  Consider a polynomial map $f: X \rightarrow \mathbb{R}^n$, where $X \subset \Real^m$ is nonempty and Zariski open (that is, $X$ can be written as
  $\Real^m \setminus K$
  where $K \neq \Real^m$ is an algebraic set).
  Then, regarding $f(X)$ as a semi-algebraic set, 
  for a generic point $\mbf{x} \in X$ we have $\dim(f(X)) = \rank\left(Df(\mbf{x})\right)$.
\end{proposition}

As we demonstrate shortly, \cref{prp:jacobian-dimension} enables a simple algorithm for computing the dimension of any set which can be specified as the image of a polynomial map.
First compute the Jacobian of the map, then evaluate it at a random input and compute its rank; with probability one this computed rank will be equal to the dimension of the image.

\subsubsection{Dimensions of matrix orbits}
\label{sec:matrix-orbit-dimension}

\begin{proposition}[Dimension of space of rank-$r$ matrices]
  \label{prp:matrix-orbit-dimension}
  The set of rank-$r$ matrices of size $m \times n$ has dimension $(m+n)r - r^2$.
\end{proposition}
\begin{proof}
  We apply \cref{prp:jacobian-dimension} to the polynomial map
  $f(\U,\V) = \U \C_r \V^{\Tr}$, which maps $\GL(m) \times \GL(n)$ onto the set of rank-$r$ matrices inside $\mathbb{R}^{m \times n}$.
 Using the product rule, the Jacobian at $(\U, \V)$ as a linear operator  from
  $\Real^{m \times m} \times \Real^{n \times n}$ to  $\Real^{m \times n}$ \nolinebreak is %
  \begin{displaymath}
    Df(\U, \V)(\dot{\U},\dot{\V}) = \dot{\U} \C_r \V^{\Tr} + \U \C_r \dot{\V}^{\Tr}\!\!,
  \end{displaymath}
  where $(\dot{\U}, \dot{\V})$ is an arbitrary element of the tangent space
  at $(\U,\V)$.
See Sections~3.4 and 4.7 in \cite{boumal2023introduction} for an intrinsic definition of Jacobians and for related computations.

  Now we show  that the rank of the Jacobian $Df(\U,\V)$ is independent of the input point $(\U,\V)$.
  Note the Jacobian may be represented as the composition of three linear maps: 
  $Df(\U,\V) = \phi \circ \gamma \circ \theta$
  where
  \begin{equation*}
    \theta(\A,\B) = (\U^{-1}\A,\B^{\Tr}\V^{-\Tr}), \quad
    \gamma(\A,\B) = \A \C_r + \C_r \B, \quad \text{and} \quad
    \phi(\mbf{M}) = \U \mbf{M} \V^{\Tr}.
  \end{equation*}
  Notice $\phi$ and $\theta$ are invertible, since $\U$ and $\V$ are invertible.  So the rank of the composite (and so the Jacobian) is equal to the rank of $\gamma$, which is independent of $\U$ and $\V$.
  Therefore, it suffices to employ \cref{prp:jacobian-dimension} at any point $(\U, \V)$.

  For simplicity, choose $(\U, \V)= (\mbf{I}_m, \mbf{I}_n)$.  
  Then
  $Df(\mbf{I}_m, \mbf{I}_n)(\dot{\U}, \dot{\V}) = \gamma(\dot{\U},\dot{\V}) = \dot{\U} \C_k + \C_k \dot{\V}^{\Tr}$.
  Writing $\dot{\U} \in \Real^{m \times m}$ and $\dot{\V} \in \Real^{n \times n}$ as block matrices
  \begin{displaymath}
    \dot{\U} =
    \begin{bmatrix} \dot{\U}_{11} & \dot{\U}_{12} \\ \dot{\U}_{21} & \dot{\U}_{22} \end{bmatrix},
    \quad
    \dot{\V} =
    \begin{bmatrix} \dot{\V}_{11} & \dot{\V}_{12} \\ \dot{\V}_{21} & \dot{\V}_{22} \end{bmatrix}
  \end{displaymath}
  with upper-left blocks of size $r \times r$, we have
  \begin{displaymath}
    Df(\mbf{I}_m, \mbf{I}_n)(\dot{\U}, \dot{\V}) =
    \begin{bmatrix} \dot{\U}_{11} & \mbf{0} \\ \dot{\U}_{21} & \mbf{0} \end{bmatrix} +
    \begin{bmatrix} \dot{\V}_{11} & \dot{\V}_{21} \\ \mbf{0} & \mbf{0} \end{bmatrix} .
  \end{displaymath}
  The range of $Df(\mbf{I}_m, \mbf{I}_n)$ consists of all matrices with a zero lower right block.
  Since this block has size $(m-r) \times (n-r)$
   and the other elements can be arbitrary nonzero numbers,
  the rank of the Jacobian is
  $mn - (m-r)(n-r) = (m+n)r - r^2$.
\end{proof}

\subsubsection{Dimension of tensor orbits}
\label{sec:tensor-orbit-dimension}
Next we turn  our attention to the dimensions of the 8 orbits of $\Real^{\ttt}$ under $\GL(2) \times \GL(2) \times \GL(2)$.
This result was first noted without proof by Kruskal in an unpublished note \cite{kruskal_rank_1983}. %
We proceed in a similar fashion to the proof of \cref{prp:matrix-orbit-dimension}: apply \cref{prp:jacobian-dimension} to a polynomial map whose Jacobian we show has the same rank at every input point, enabling us to
choose a convenient input (identity matrices) to simplify the analysis.  

\begin{theorem}
  \label{thm:tensor-orbit-dimension}
  The dimensions of the 8 orbits of $\mbb{R}^{\ttt}$ under $\GL(2) \times \GL(2) \times \GL(2)$ are as given in \cref{tab:tensor-orbits}.
\end{theorem}
\begin{proof}
  Let $\Tn{G}$ be one of the canonical forms from \cref{tab:tensor-orbits}
  and consider the map
  $f : \GL(2) \times \GL(2) \times \GL(2) \rightarrow \Real^{\ttt}$, defined as
  \begin{displaymath}
    f(\U, \V, \W) = \Tn{G} \times_1 \U \times_2 \V \times_3 \W.
  \end{displaymath}
  In particular, $f$ maps from $\Real^{12}$ to $\Real^8$ if we vectorize
  the inputs and outputs.
  An expression for the Jacobian is
  \begin{displaymath}
    Df(\U,\V,\W)(\dot{\U}, \dot{\V}, \dot{\W}) =
    \Tn{G} \times_1 \dot{\U} \times_2      \V  \times_3      \W +
    \Tn{G} \times_1      \U  \times_2 \dot{\V} \times_3      \W +
    \Tn{G} \times_1      \U  \times_2      \V  \times_3 \dot{\W}.
  \end{displaymath}
  This expression  gives $Df(\U,\V,\W)$ as a linear operator from $\Real^{12}$ to $\Real^8$.

  We next show that the rank of the Jacobian is independent of $\U, \V, \W$.
  We express it as the composition of three maps:
  $Df(\U, \V, \W) 
   = \phi \circ \gamma \circ \theta$
  where
  \begin{align*}
    \theta(\A, \B, \C) &= (\U^{-1}\A,\V^{-1}\B,\W^{-1}\C), \\
    \gamma(\A,\B,\C) 
    &=     \Tn{G} \times_1 \A        \times_2 \mbf{I}_2 \times_3 \mbf{I}_2 +
          \Tn{G} \times_1 \mbf{I}_2 \times_2 \B        \times_3 \mbf{I}_2 
          \Tn{G} \times_1 \mbf{I}_2 \times_2 \mbf{I}_2 \times_3 \C, 
          \quad\text{and}\\
    \phi(\T) &= \T \times_1 \U \times_2 \V \times_3 \W .
  \end{align*}
  Notice $\phi$ and $\theta$ are invertible, since $\U, \V, \W$ are invertible. 
  So the rank of the composition (and so the Jacobian) is equal to the rank of $\gamma$, which is independent of $\U, \V, \W$.
  Hence, for simplicity we may choose the point $(\U,\V,\W)=(\mbf{I}_2, \mbf{I}_2, \mbf{I}_2)$ for applying \cref{prp:jacobian-dimension}, because then
$Df(\mbf{I}_2,\mbf{I}_2,\mbf{I}_2) = \gamma$.

  It remains only to calculate the rank of this linear map for each orbit.
  To recover the matrix representation of the linear map from its action as an operator, one can plug in a basis for its domain.
  Since the input is 3 matrices of size $2 \times 2$, 
  it has dimension 12.
  If one chooses the basis to be the elementary vectors,
  we get one output per basis vector.
  For example, consider the orbit $D_1$ with canonical form
  \begin{displaymath}
    \G = 
    \begin{tensoreight}
      1 & 0 & 0 & 0 \\
      0 & 0 & 0 & 0
    \end{tensoreight}.
  \end{displaymath}
  Evaluating the first ``unit vector'' in $\Real^{12}$ is equivalent to
  \begin{displaymath}
    Df(\mbf{I}_2,\mbf{I}_2,\mbf{I}_2)\left(
    \begin{bmatrix}
      1 & 0 \\ 0 & 0
    \end{bmatrix},
    \begin{bmatrix}
      0 & 0 \\ 0 & 0
    \end{bmatrix},
    \begin{bmatrix}
      0 & 0 \\ 0 & 0
    \end{bmatrix}
    \right) =
    \begin{tensoreight}
      1 & 0 & 0 & 0 \\
      0 & 0 & 0 & 0
    \end{tensoreight}.
  \end{displaymath}
    Vectorizing this yields the first column of the Jacobian.
  Continuing in this fashion with the other 11 basis vectors, we can recover the full Jacobian for orbit $D_1$ as
  \begin{align*}
    \begin{bmatrix}
      1 & 0 & 0 & 0 & 1 & 0 & 0 & 0 & 1 & 0 & 0 & 0 \\
      0 & 0 & 0 & 0 & 0 & 0 & 0 & 0 & 0 & 0 & 1 & 0 \\
      0 & 0 & 0 & 0 & 0 & 0 & 1 & 0 & 0 & 0 & 0 & 0 \\
      0 & 0 & 0 & 0 & 0 & 0 & 0 & 0 & 0 & 0 & 0 & 0 \\
      0 & 0 & 1 & 0 & 0 & 0 & 0 & 0 & 0 & 0 & 0 & 0 \\
      0 & 0 & 0 & 0 & 0 & 0 & 0 & 0 & 0 & 0 & 0 & 0 \\
      0 & 0 & 0 & 0 & 0 & 0 & 0 & 0 & 0 & 0 & 0 & 0 \\
      0 & 0 & 0 & 0 & 0 & 0 & 0 & 0 & 0 & 0 & 0 & 0
    \end{bmatrix}\!,
  \end{align*}
  which is rank 5 by inspection.
  Hence $\dim(D_1) = 5$ as claimed, with the dimensions of the other orbits following by an analogous computation involving the respective canonical form; see \cref{sec:tensor-orbit-details} for the other 7 Jacobians as well as an alternative proof.
\end{proof}

\Cref{thm:tensor-orbit-dimension} justifies the dimension column
of \cref{tab:tensor-orbits}.
We now have a better understanding of the topology of different ranks: the set of rank-2 tensors has a full-dimensional component ($G_2$) and a low-dimensional component ($D_2 \cup D_2^\prime \cup D_2^{\prime\prime}$); likewise, the full and low-dimensional components of set of rank-3 tensors are $G_3$ and $D_3$.
The full-dimensional sets are why there are multiple typical ranks, and the low-dimensional rank-2 and rank-3 sets have an important role to play in the discussion of ill-posedness in the next section.

Before we move on, we consider the relative sizes of $G_2$ and $G_3$.
Both have dimension 8, but is there any difference between them?
It turns out that random matrix theory enables the calculation of the exact proportion of space occupied by these two orbits.
Analytical results by Bergqvist \cite{bergqvist_exact_2013} show that
the probability that a random tensor with i.i.d.\@ standard normal entries being
in $G_2$ is $\pi/4$ and the probability of being in $G_3$ is $1-\pi/4$.
This corresponds to the last column of \cref{tab:tensor-orbits}.
The proof shows an equivalence between the rank of a $2 \times 2 \times 2$ tensor and the number of real generalized eigenvalues, after which a result from random matrix theory about the expected number of such eigenvalues is employed.
The following exercise explores the probabilities of rank-2 and rank-3 tensors.

\begin{exercise}[Probabilities of rank-2 and rank-3]
  By examining the sign of the hyperdeterminant of a tensor and its multilinear rank, \cref{tab:tensor-orbits} gives an algorithm for determining the rank of a tensor.
  Generate 1000 random $\ttt$ tensors with standard Gaussian i.i.d.\@ entries
  and determine the rank for each. 
  \begin{enumerate}[label={(\roman*)}]
  \item What is the empirical probability of a tensor being rank-2? What about rank-3?
    How do these compare to the probabilities listed in \cref{tab:tensor-orbits}?
  \item Did you encounter any tensors from degenerate (low dimensional) orbits?
  \end{enumerate}
\end{exercise}

\section{Tesselation and topology}
\label{sec:tesselation-topology}
To answer the questions regarding tesselation and topology it will be helpful to use the concept of a partial ordering.
A \emph{partial ordering}  is a binary relation $\preceq$ satisfying the following properties:
reflexivity ($A \preceq A$),
transitivity ($A \preceq B$ and $B \preceq C \Rightarrow A \preceq C$), and
antisymmetry ($A \preceq B$ and $B \preceq A \Rightarrow A = B$).
A partially ordered set is simply a set with a partial ordering.
Since the set is only partially ordered, not every pair of elements need be comparable via $\preceq$; for $A, B \in S$ it is permitted that neither $A \preceq B$ nor $B \preceq A$ is true.

\begin{example}\label{exa:partial-order}
  The set
  $\set{ \set{x,y,z}, \set{x,y}, \set{y,z}, \set{x}, \set{y}, \set{z}, \set{} }$
  can be partially ordered by inclusion, i.e., $A \preceq B$ if $A \subseteq B$.
  Note that $\set{x,y}$ and $\set{y,z}$ are not comparable under this relation.
\end{example}

A \emph{Hasse diagram} is a visual representation of a partially ordered set as a directed graph \cite{birkhoff_lattice_1967}.
If $A \preceq B$, then the Hasse diagram has an arrow from $A$ to $B$ ($A \rightarrow B$),
excluding edges implied by transitivity.
While set inclusion can induce a partial ordering as in \cref{exa:partial-order}, this relation is not well suited for orbits, which are always disjoint.
Instead, one of the most natural partial orderings for orbits is the inclusion of closures; that is
$A \preceq B$ if $\overline{A} \subseteq \overline{B}$.
The following example uses this relation to partially order certain sets in $\Real^2$.

\begin{example}
  Consider a partitioning of $[0,2] \times [0,2] \subset \Real^2$ into five sets (defined below), partially ordered by inclusion of closures.
  The corresponding Hasse diagram is shown in \cref{fig:hasse-example}.
\end{example}
\begin{figure}[!ht]
  \centering
  \includegraphics{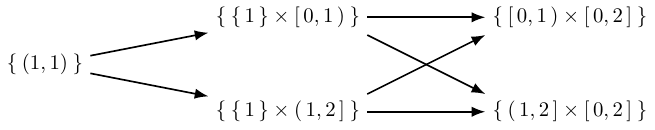}
  \caption{Hasse diagram for a partitioning of
    $[0,2] \times [0,2] \subset \Real^2$
    partially ordered by inclusion of closures.}
  \label{fig:hasse-example}
\end{figure}

\subsection{Matrix Hasse diagram}

Consider partially ordering (by inclusion of closures) the orbits of $\GL(m) \times \GL(n)$ in $\Real^{m \times n}$.
As discussed in \cref{sec:orbit-matrix}, there are $n + 1$ total orbits (assuming $m \geq n$).
The Hasse diagram then has $n+1$ elements as shown in \cref{fig:hasse-matrix}.
We use $R_r$ to denote the orbit
$\left( \GL(m) \times \GL(n) \right) \cdot \C_r$, which consists of all rank-$r$ matrices.
\begin{figure}[!ht]
  \centering
  \includegraphics{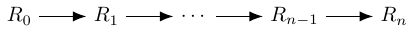}
  \caption{Hasse diagram of orbits in $\Real^{m \times n}$ under $\GL(m) \times \GL(n)$, ordered by inclusion of closures.}
  \label{fig:hasse-matrix}
\end{figure}

To prove the correctness of each edge in the diagram it suffices to verify that
rank-$(r-1)$ matrices are in the closure of rank-$r$ matrices.
Consider the sequence
\begin{displaymath}
  \label{eqn:matrix-rank-drop}
  \C_{r-1} + \frac{1}{j} \C_r
  \in \Real^{m \times n}
  \quad \text{for} \quad j \in \mathbb{N}.
\end{displaymath}
Every matrix in the sequence has rank $r$, but the limit of the sequence is
$\C_{r-1}$ of rank $r-1$.

We conclude the example by noting a few simple facts made apparent by the Hasse diagram.
Only rank-$0$ matrices form a closed set; orbits of higher rank are missing the limit points corresponding to all lower matrix ranks.
While matrices of \emph{fixed} rank do not form a closed set, matrices of \emph{bounded} rank do. In other words, for $r > 0$, the set
$\set{\A \in \Real^{m \times n} | \rank(\A) = r}$ is not closed, but the set
$\set{\A \in \Real^{m \times n} | \rank(\A) \leq r}$ is.

\subsection{Tensor Hasse diagram}
As in the matrix case above, the tensor orbits from \cref{tab:tensor-orbits} may be put into a partial order via the inclusion of closures to provide geometrical information and intuition about these sets.
The resulting Hasse diagram for the tensor orbits is shown in \cref{fig:hasse-tensor}.
Like the proof of the tensor orbit dimensions, this seems to have been first done by Kruskal without written proof in an unpublished note \cite{kruskal_rank_1983},
so we provide a proof in \cref{thm:hasse-tensor}.

\begin{figure}[!ht]
  \centering
  \includegraphics{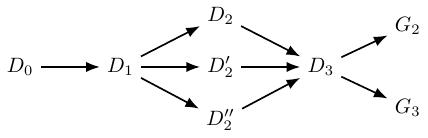}
  \caption{
    Hasse diagram of the orbits of $\Real^{\ttt}$ under $\GL(2) \times \GL(2) \times \GL(2)$, ordered by inclusion of closures.
  }
  \label{fig:hasse-tensor}
\end{figure}

\begin{theorem}
  \label{thm:hasse-tensor}
  The Hasse diagram given in \cref{fig:hasse-tensor} describes the partial ordering of the orbits of $\ttt$ tensors under $\GL(2) \times \GL(2) \times \GL(2)$.
\end{theorem}
\begin{proof}
  To justify an edge corresponding to $\overline{A} \supset \overline{B}$ in the diagram it is sufficient to prove the following: there exists a sequence
  $( \U^{(j)}, \V^{(j)}, \W^{(j)} )_{j \in \mbb{N}}
  \subset \GL(2) \times \GL(2) \times \GL(2)$
  such that
  \begin{align}
    \lim_{j \rightarrow \infty}
    \Tn{G}_A \times_1 \U^{(j)} \times_2 \V^{(j)} \times_3 \W^{(j)}
    = \Tn{G}_B
  \end{align}
  where $\Tn{G}_i$ denotes the canonical form of orbit $i$.
  We only prove there exists an edge corresponding to
  $\overline{D_3} \supset \overline{D_2}$,
  as the other cases are nearly identical.
  We select the sequence
  \begin{align}
    (\U^{(j)},\V^{(j)},\W^{(j)}) =
    \left(
      \begin{bmatrix} 0 & 1 \\ \frac{1}{j} & 0 \end{bmatrix},
      \begin{bmatrix} 0 & 1 \\ 1 & 0 \end{bmatrix},
      \mathbf{I_2}
    \right) ,
  \end{align}
  which yields
  \begin{multline*}
      \lim_{j \rightarrow \infty}
      \begin{tensoreight}
        1 & 0 & 0 & 0 \\
        0 & 1 & 1 & 0
      \end{tensoreight}
      \times_1 \begin{bmatrix} 0 & 1 \\ \frac{1}{j} & 0 \end{bmatrix}
      \times_2 \begin{bmatrix} 0 & 1 \\ 1 & 0 \end{bmatrix}
      \times_3 \mbf{I}_2 \\
      =
      \lim_{j \rightarrow \infty}
      \begin{tensoreight}
        1 & 0 & 0 & 1 \\
        0 & \frac{1}{j} & 0 & 0 \\
      \end{tensoreight}
      =
      \begin{tensoreight}
        1 & 0 & 0 & 1 \\
        0 & 0 & 0 & 0
      \end{tensoreight}.
   \end{multline*}
\end{proof}

Consider one consequence of the Hasse diagram in \cref{fig:hasse-tensor}:
the 7-dimensional orbit $D_3$ is in the closure of both 8-dimensional orbits, $G_2$ and $G_3$.
Indeed, the closure of $D_3$ is the boundary between these two orbits.
Another way to see this is in terms of the sign of the hyperdeterminant: positive for $G_2$, negative for $G_3$, and zero for the closure of $D_3$.
Hence, by continuity of the hyperdeterminant, $D_3$ must lie between $G_2$ and $G_3$.

\subsection{Visualizations}\label{sec:visualizations}
We conclude this section with computational visualizations the geometry of ranks for matrices and tensors.
This will build intuition, informing our
discussion regarding \emph{ill-posedness} in the next section.
We have created an interactive tool at \webtool{}
that enables the user to zoom in and rotate the 3D visualizations displayed in \cref{fig:affine-intro,fig:matrix-affine,fig:generic,fig:D2pp}.
A user can alternatively specify their own center point 
and axes defining an affine slice to generate different renderings.
This tool is built on top of Spirulae.\footnote{%
 Spirulae: Web-Based Math Visualization is an open-source tool by Harry Chen,
 available at
 \href{https://harry7557558.github.io/spirulae/}{\texttt{https://harry7557558.github.io/spirulae/}}.}
Our website contains usage documentation, as well as links to the original and modified source code.
For those curious for more technical details, our README explains how the visualizer works and how it was adapted from the original tool.

\subsubsection{Matrix visualizations}
We begin with a $3 \times 3$ matrix case.
We cannot visualize the full 9-dimensional space of $3 \times 3$ matrices or the hypersurface where the determinant vanishes,
but we can visualize 2- and 3-dimensional affine slices through this space.
On the left of \cref{fig:matrix-affine},
we visualize an affine slice of the form
\begin{displaymath}
  \set{ \mbf{M} : \mbf{M} = \mbf{M}_0 + \lambda_1 \Mx{M}_1 + \lambda_2 \Mx{M}_2 }.
\end{displaymath}
The matrix $\Mx{M}{0}$ is the center point, and in this case we have selected
a rank-1 matrix so that in particular $\det(\Mx{M}{0})=0$. The matrices $\Mx{M}{1}$ and $\Mx{M}{2}$ form the axes and are chosen so that their vectorizations are orthogonal
and of unit length.

In three dimensions, we visualize the surface $\det(\Mx{M})=0$ where 
\begin{displaymath}
  \set{ \mbf{M} : \mbf{M} = \mbf{M}_0 + \lambda_1 \Mx{M}_1 + \lambda_2 \Mx{M}_2 + \lambda_3 \Mx{M}_3 }.
\end{displaymath}
The matrix $\Mx{M}{3}$ is the third axis, and it is of unit length and 
orthogonal to the $\Mx{M}{1}$ and $\Mx{M}{2}$ matrices previously used.
The ``empty'' space in the 3-dimensional visualization corresponds to rank-3 matrices.
See the right panel of \cref{fig:matrix-affine}.

\begin{figure}[!ht]
  \centering
  \resizebox{\textwidth}{!}{
    \includegraphics[height=1cm]{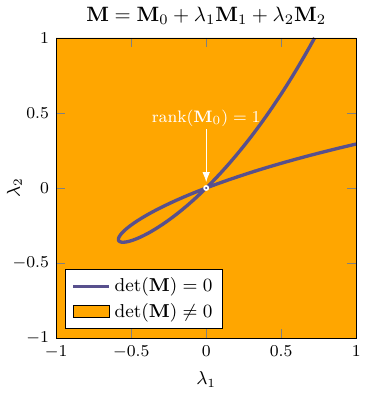}
    \includegraphics[height=1cm]{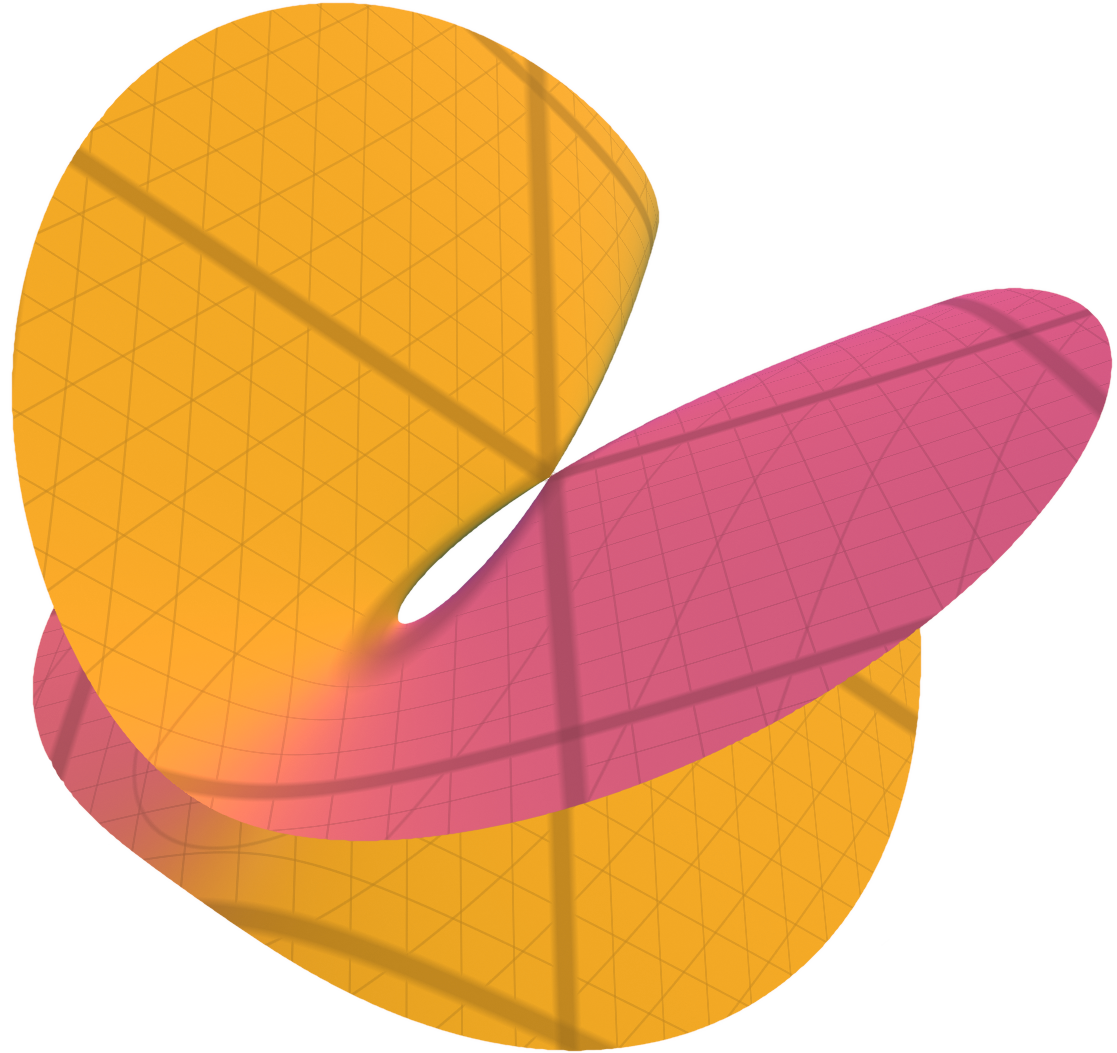}
    }
    \caption{Affine slices through $\Real^{3 \times 3}$ centered at a rank-1 matrix. The coordinate axes correspond to matrices that are pairwise orthogonal and of unit length.
    On the right, the visualized surface shows where the determinant is $0$.
  The lines  indicate the intersection of the surface with a Cartesian grid and are shown only to guide the eye.}
    \label{fig:matrix-affine}
  \end{figure}

The 2-dimensional slice of \cref{fig:matrix-affine} illustrates how rank-2 matrices are in the closure of rank-3 matrices.
Likewise, the center point is a rank-1 matrix in the closure of the rank-2 matrices.

\subsubsection{Tensor visualizations}
Now we turn our attention to analogous visualizations for $\ttt$ tensors.
We cannot view the 8-dimensional space of $\ttt$ tensors or the hypersurface in $\Real^{\ttt}$ where the hyperdeterminant vanishes, so we again visualize 2- and 3-dimensional affine slices through this space.
In the 2D visualization, we show the affine slice defined by
\begin{displaymath}
  \set{ \T : \T = \T_0 + \lambda_1 \T_1 + \lambda_2 \T_2 }.
\end{displaymath}
We will use different choices for $\T_0$ in the different visualizations.
Tensors $\T_1$ and $\T_2$ are the axes and are chosen so that their vectorizations are orthogonal and of unit length.
In the two-dimensional visualizations, we distinguish between $G_2$ and $G_3$ tensors by coloring the regions, with $\Delta(\T)=0$ defining the
border between the two sets.  
Here, $\Delta(\T)=0$ are so-called plane quartic curves, which are classically studied curves in algebraic geometry \cite{dolgachev2012classical}.  
In the 3-dimensional visualizations, we show the surface $\Delta(\T)=0$, and coloring is only related to the normal direction of the surface to increase visual contrast and add depth.
The third axis is chosen to be orthogonal to the first two and of unit length.
``Empty'' points in the 3-dimensional visualization correspond to $G_2$ and $G_3$ tensors, where $\Delta(\T) \neq 0$.

In \cref{fig:generic}, the center point $\T_0$ is a generic tensor of rank 3.
On the left is a 2D slice through $\Real^{\ttt}$, and on the right is a 3-dimensional visualization of the surface $\Delta(\T)=0$.
In the 2D version, the curve where the hyperdeterminant is zero divides the two positive volume regions, as we deduced from the Hasse diagram in \cref{fig:hasse-tensor}.
The 3D visualization captures singularities not seen in the 2D visualization.
This is no coincidence: since $D_2 \cup D_2^\prime \cup D_2^{\prime\prime}$ has dimension 5 in the ambient space of dimension 8, these orbits are expected to have intersections with a generic affine space of dimension 3; the key here is that the sum of the dimensions of the surface and the affine space is greater than or equal to the dimension of the ambient space.
The same is not true of $D_1$ as its dimension (4) is too low.

\begin{figure}
  \centering
  \resizebox{\textwidth}{!}{
    \includegraphics[height=1cm]{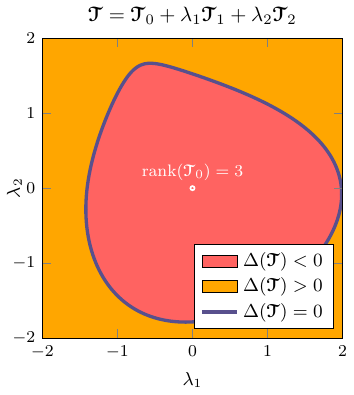}
    \includegraphics[height=1cm]{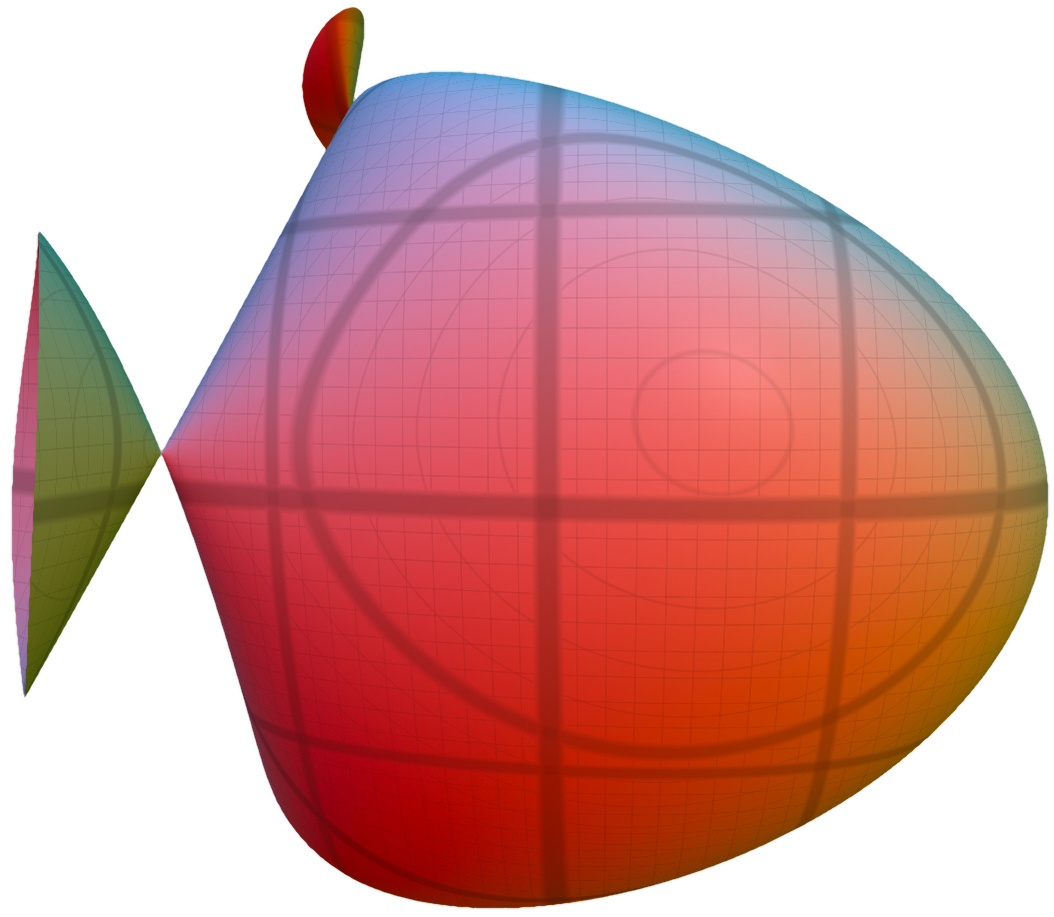}
  }
  \caption{Affine slices through $\Real^{\ttt}$ centered at a tensor in $G_3$. The coordinate axes correspond to tensors that are pairwise orthogonal and of unit length.
  On the right, the visualized surface shows where the hyperdeterminant is $0$.
  The lines  indicate the intersection of the surface with a Cartesian grid and are shown only to guide the eye.
    }
  \label{fig:generic}
\end{figure}

In \cref{fig:D2pp} we visualize the case where $\T_0 \in D_2^{\prime\prime}$, and the behavior is quite different.
The center point $\T_0$ now sits on the separating hypersurface between the two positive volume rank-2 and rank-3 components.
Furthermore, it is apparent that the geometry is more complex around such points than for example the random $G_3$ point in \cref{fig:generic}.
Indeed, all orbits below $D_3$ in \cref{fig:hasse-tensor} are \emph{singularities} of the separating hypersurface.
Algebraically, this means that the \emph{gradient} of the hyperdeterminant is $0$ at such points (where the hyperdeterminant is already 0) \cite[Section~3.3]{bochnak_real_1998}.

\begin{figure}
  \centering
  \resizebox{\textwidth}{!}{
    \includegraphics[height=1cm]{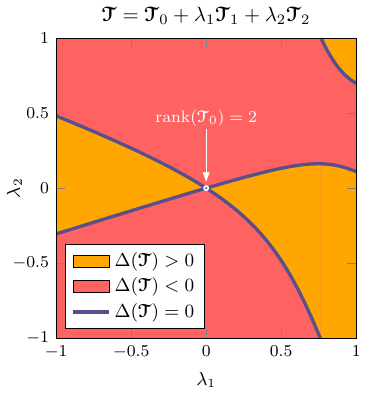}
    \includegraphics[height=1cm]{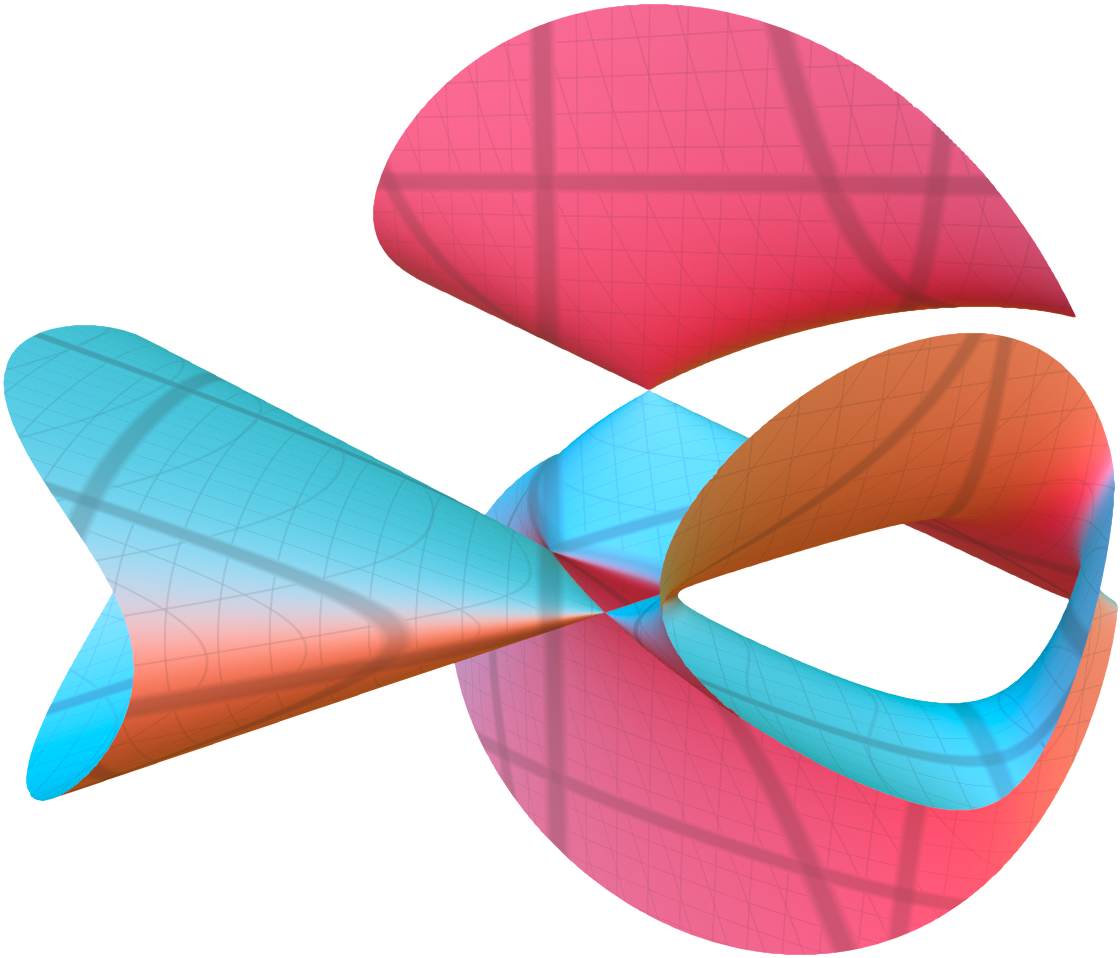}
  }
  \caption{
    Affine slices through $\Real^{\ttt}$ centered at a tensor in $D_2^{\prime\prime}$. The coordinate axes correspond to tensors that are pairwise orthogonal and of unit length.
    On the right, the visualized surface shows where the hyperdeterminant is $0$.
  The lines  indicate the intersection of the surface with a Cartesian grid and are shown only to guide the eye.
    }
  \label{fig:D2pp}
\end{figure}

As further  evidence that the hyperdeterminant zero hypersurface is singular on orbits below $D_3$, consider \cref{fig:affine-intro}.
In this figure the 3-dimensional affine slice is centered around the canonical form of $D_1$, which has rank 1.
These singularities will play a
role in the ill-posedness of low-rank approximation for tensors, discussed next.

\section{Low-rank approximation and ill-posedness}
\label{sec:ill-posedness}
In this section, we explore the problems of low-rank matrix and tensor approximation.
For our purposes, we say a problem \emph{ill-posed} if no solution exists (or is attained)
and \emph{well-posed} if it has a solution.
The matrix low-rank approximation problem is known to be well-posed according to the Eckart-Young theorem \cite{eckart_approximation_1936}.
Unfortunately, the tensor low-rank approximation problem may be ill-posed.

In the tensor case, the low-rank approximation problem is as follows.
Given a tensor $\Tn{A}$ (possibly of rank greater than $r$), we seek a tensor $\Tn{B}$ of rank at most $r$ that is ``closest'' to $\Tn{A}$ in the sense of minimizing the distance:
\begin{equation}\label{eqn:tensor-low-rank}
  \min_{\rank(\Tn{B}) \leq r} \frac{1}{2}\norm{\Tn{A} - \Tn{B}}^2.
\end{equation}
Here, the norm is induced from vectorization, equal to the
square root of the sum of the squares of the entries.
Since the tensor low-rank approximation problem can fail to have a solution (as we will show shortly), it must be that the set of tensors with rank $\leq r$ is not closed.
Equivalently, and unlike matrix rank, tensor rank is not lower-semicontinuous \cite{de_silva_tensor_2008}.\footnote{The referenced paper mistakenly uses ``upper semicontinuous'' in places where ``lower semicontinuous'' is intended and correct.}
For $\ttt$ tensors we will show that many boundary points are missing.  For those boundary points which do exist, the local geometry pathologically prevents them being solutions to low-rank approximation problems.
From another point of view, if we replace the feasible domain in problem \eqref{eqn:tensor-low-rank} by the closure, then the minimum can be attained at a tensor of rank strictly greater than $r$, and such a tensor can be written as the limit of some sequence of rank $\leq r$ tensors.

\subsection{Border rank}

The concepts of closedness, rank, and sequences of bounded rank tensors are nicely connected by the last notion of tensor rank we consider.
We say a tensor $\Tn{T}$ has \emph{border rank} $r$ if $r$ is the smallest integer such that, for any $\epsilon>0$, there exists a rank-$r$ tensor $\Tn{T}_\epsilon$ with $\norm{\Tn{T} - \Tn{T}_\epsilon} \leq \epsilon$.
Informally, a tensor is border rank $r$ if it may be approximated to arbitrary accuracy by rank-$r$ tensors.
The border rank of a tensor is always less than or equal to its rank.

\begin{exercise}[Multilinear rank lower bounds border rank]
  \label{exr:mrank-brank}
 Prove that \cref{exr:mrank-rank} extends to border rank as well, that is show that the maximal flattening rank is a lower bound for the border rank.
 \textit{Hint:} The closure of the set of rank-$r$ tensors includes all border rank-$r$ tensors, while the set of tensors with all flattening ranks $\leq r$ is already a closed set.
\end{exercise}

\begin{exercise}
  From the Hasse diagram in \cref{fig:hasse-matrix} one can see that the 
  closure of rank-$r$ matrices contains all matrices of lower rank. 
  (a) Prove that border rank is always equal to the rank for a matrix.
  (b) Use the Hasse diagram in \cref{fig:hasse-tensor} to argue that the same thing is not true for tensors.
\end{exercise}

\begin{exercise}
  Consider the the canonical form of $G_2$,
  $\begin{tensoreight}[2ex]
      1 & 0 & 0 & 0 \\
      0 & 0 & 0 & 1
    \end{tensoreight}$.
  This tensor has the rank-2 decomposition
  $\begin{bsmallmatrix} 1 \\ 0 \end{bsmallmatrix} \circ
   \begin{bsmallmatrix} 1 \\ 0 \end{bsmallmatrix} \circ
   \begin{bsmallmatrix} 1 \\ 0 \end{bsmallmatrix} +
   \begin{bsmallmatrix} 0 \\ 1 \end{bsmallmatrix} \circ
   \begin{bsmallmatrix} 0 \\ 1 \end{bsmallmatrix} \circ
   \begin{bsmallmatrix} 0 \\ 1 \end{bsmallmatrix}$.
   Let $\T$ be any rank-1 tensor.
   Prove that the elements of the sequence
  $\begin{bsmallmatrix} 1 \\ 0 \end{bsmallmatrix} \circ
   \begin{bsmallmatrix} 1 \\ 0 \end{bsmallmatrix} \circ
   \begin{bsmallmatrix} 1 \\ 0 \end{bsmallmatrix} +
   \begin{bsmallmatrix} 0 \\ 1 \end{bsmallmatrix} \circ
   \begin{bsmallmatrix} 0 \\ 1 \end{bsmallmatrix} \circ
   \begin{bsmallmatrix} 0 \\ 1 \end{bsmallmatrix} +
   \frac{1}{k}\T$
  are not rank-$3$ for sufficiently large $k$.
  That is, prove the existence of $\hat{k}$ such that the sequence elements are rank-$2$ for all $k \geq \hat{k}$.
  This demonstrates you cannot construct sequences which drop rank only in the limit, as one can for the matrix case (\cref{eqn:matrix-rank-drop}).
\end{exercise}

To see that, unlike matrices, tensors can have different rank and border rank we prove that all $D_3$ tensors (which are rank-3) have border rank 2.

\begin{theorem}[$D_3$ border rank \cite{bini_approximate_1980,de_silva_tensor_2008}]
  \label{thm:d3-brank}
  Tensors in orbit $D_3$ have border rank 2.
\end{theorem}
\begin{proof}
  Let $\Vc{x}{1}$, $\Vc{x}{2}$, $\Vc{x}{3}$, $\Vc{y}{1}$, $\Vc{y}{2}$, and $\Vc{y}{3}$ be vectors in $\Real^2$ such that $\set{\Vc{x}{i}, \Vc{y}{i}}$ are
  linearly independent for all $i$.
  For $k \in \mathbb{N}$, define
  \begin{align}
    \label{eqn:rank-jump-sequence-1}
    \T_k &= k \left(\mbf{x}_1 + \frac{1}{k} \mbf{y}_1\right) \circ
    \left(\mbf{x}_2 + \frac{1}{k} \mbf{y}_2\right) \circ
    \left(\mbf{x}_3 + \frac{1}{k} \mbf{y}_3\right)
    - k \mbf{x}_1 \circ \mbf{x}_2 \circ \mbf{x}_3 \\
    \label{eqn:rank-jump-sequence-2}
    &=
    \begin{aligned}[t]
      &\mbf{y}_1 \circ \mbf{x}_2 \circ \mbf{x}_3 +
      \mbf{x}_1 \circ \mbf{y}_2 \circ \mbf{x}_3 +
      \mbf{x}_1 \circ \mbf{x}_2 \circ \mbf{y}_3 \\
      &+ \frac{1}{k}\left(
        \mbf{x}_1 \circ \mbf{y}_2 \circ \mbf{y}_3 +
        \mbf{y}_1 \circ \mbf{x}_2 \circ \mbf{y}_3 +
        \mbf{y}_1 \circ \mbf{y}_2 \circ \mbf{y}_3
        \right) +
      \frac{1}{k^2}\left(\mbf{y}_1 \circ \mbf{y}_2 \circ \mbf{y}_2\right), \nonumber
    \end{aligned}
  \end{align}
  a sequence of tensors originally from \cite{bini_approximate_1980}.
  Letting $\Tn{T}^* = \lim_{k \rightarrow \infty} \T_k$ it is clear that the border rank of $\T^*$ is $\leq 2$ since $\T_k$ has rank at most 2 from \cref{eqn:rank-jump-sequence-1}.
  
  We will show that $\Tn{T}^*$ can be any $D_3$ tensor.
  If we set
  $\mbf{x}_2 = \mbf{x}_3 = \mbf{y}_1 = 
  \begin{bsmallmatrix} 1 \\ 0 \end{bsmallmatrix}$
  and
  $\mbf{x}_1 = \mbf{y}_2 = \mbf{y}_3 = 
  \begin{bsmallmatrix} 0 \\ 1 \end{bsmallmatrix}$,
  then $\Tn{T}^*$ is the canonical form of $D_3$.
  Then an action by an element of $\GL(2) \times \GL(2) \times \GL(2)$ can transform to any other $D_3$ tensor by definition of orbits; since this action can be written as a transformation of the vectors $\mbf{x}_i, \mbf{y}_i$ defining the sequence we can construct the sequence for any $D_3$ tensor.
  Hence all $D_3$ tensors have border rank $\leq 2$.
  To conclude that $D_3$ tensors have border rank $\geq 2$, note that the multilinear rank of the canonical form of $D_3$ is $(2,2,2)$ by inspection.
  Since multilinear rank lower bounds border rank (\cref{exr:mrank-brank}) and border rank is invariant under  $\GL(2) \times \GL(2) \times \GL(2)$,  all $D_3$ tensors have border rank 2.
\end{proof}

\subsection{Optimality and cones}
\label{sec:cones-optimality}
We use tools from optimization theory to reason about the existence of solutions
for low-rank approximation problems.
Consider a general optimization problem that seeks to minimize a differentiable function
$f : \Real^m \rightarrow \Real$
over a feasible set $\Omega \subset \Real^m$.
We can state the first-order necessary conditions in terms of the geometry of
the feasible region by relating the direction of the gradient with the tangent cone.

A \emph{cone}, $K \subset \Real^m$, is a set of vectors that is closed under nonnegative scalar multiplication.
Its \emph{polar cone}, denoted $K^\circ$, is the set of vectors that form an angle of at least 90 degrees with every vector in $K$, i.e.,
$K^\circ =
\set{ \mbf{v} \in \Real^m : \mbf{v}^{\Tr} \mbf{w} \leq 0 \; \, \forall \, \mbf{w} \in K }$.

\begin{exercise}\label{exr:polar-cone}
  Show that if $K_1$ and $K_2$ are cones with $K_1 \subset K_2$, then $K_2^\circ \subset K_1^\circ$.
\end{exercise}

Informally, the \emph{tangent cone} of a set $\Omega$ at a  point $\mbf{x}$ in the set,
denoted $T_{\Omega}(\mbf{x})$, can be thought of as the limit of the set of $\epsilon$-feasible directions.
The formal definition is given below, and an illustration can be found in \cref{fig:tangent-normal-cones}.

\begin{definition}
  [Bouligand tangent cone {\cite[Definition 3.11]{ruszczynski_nonlinear_2006}}]
  \label{def:tangent-cone}
  The \emph{tangent cone} of a set $\Omega$ at a point $\mbf{x}$ in $\Omega$  is
  \begin{displaymath}
    T_{\Omega}(\mbf{x})
    = \Set{
       \lim_{k \rightarrow \infty} \frac{\mbf{x}_k - \mbf{x}}{\tau_k}
      :
    \set{ \mbf{x}_k } \subset \Omega \text{ with } \mbf{x}_k \rightarrow \mbf{x} \text{ and }
    \set{ \tau_k } \subset \mbb{R}_{>0} \text{ with } \tau_k \downarrow 0}.
  \end{displaymath}
\end{definition}

The polar cone of the tangent cone is called the \emph{normal cone}.  It is given by
\begin{align}
  N_{\Omega}(\mbf{x}) = \left( T_{\Omega}(\mbf{x}) \right)^\circ =
  \set{ \mbf{v} \in \Real^m : \mbf{v}^{\Tr} \mbf{w} \geq 0 \; \, \forall \, \mbf{w} \in T_{\Omega}(\mbf{x}) } .
\end{align}

With tangent and normal cones defined, we are in a position to state the first-order optimality condition precisely.

\begin{theorem}[Necessary first-order optimality condition {\cite[Theorem 3.24]{ruszczynski_nonlinear_2006}}]
  \label{thm:first-order-necessary}
  If $\mbf{x}^*$ is a local minimizer of the problem
  $\min_{\mbf{x} \in \Omega} f(\mbf{x})$
  then
  \begin{align*}
    -\nabla f(\mbf{x}^*) \in N_{\Omega}(\mbf{x}^*).
  \end{align*}
\end{theorem}

\begin{figure}[!ht]
  \centering
  \includegraphics{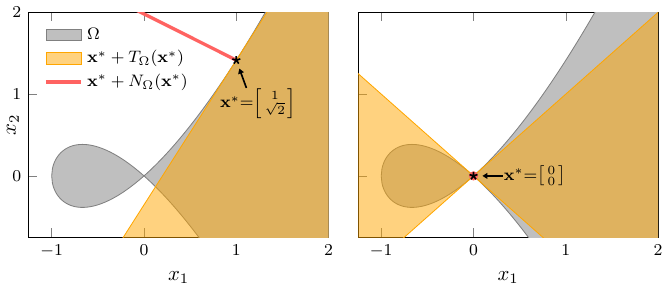}
  \caption{Tangent and normal cones with respect to
    $\Omega \equiv \set{\mbf{x} \in \Real^2 : x_1^2+x_1^3-x_2^2 \geq 0}$.}
  \label{fig:tangent-normal-cones}
\end{figure}

\Cref{fig:tangent-normal-cones} shows examples of tangent and normal
cones for $\Omega \in \Real^2$.
The feasible space is the set
$\Omega\equiv\set{\mathbf{x} \in \Real^2 : x_1^2+x_1^3-x_2^2 \geq 0}$.
In the left plot,
the tangent cone is a half space and the normal cone is a half line.
In the right plot,
the tangent cone is more complex and the normal cone is just $\set{\mbf{0}}$.

\begin{exercise}
  For $\Omega$ in \cref{fig:tangent-normal-cones}, what are $T_{\Omega}(\mbf{x}^*)$ and $N_{\Omega}(\mbf{x}^*)$ at
  $\mbf{x}^* = \begin{bsmallmatrix} -1 \\ 0 \end{bsmallmatrix}$?
\end{exercise}

Computing complete tangent or normal cones can be difficult in practice, making it difficult to check if the necessary condition in \cref{thm:first-order-necessary} is satisfied.
However, the following basic results provide some leverage.

\begin{proposition}[{\cite[Lemma 3.12]{levin_effect_2025}}]
  \label{prp:linear-subspace-tangent}
  Let $\Omega \subset \Real^m$ and consider any differentiable map $p: \Real^n \rightarrow \Real^m$ with image contained in $\Omega$.
  Then for any $\mbf{x} \in \Real^m$, the column space of the Jacobian of $p$ at $\mbf{x}$ lies in the tangent cone of $\Omega$ at $p(\mbf{x})$, i.e.,
  $\col(Dp(\mbf{x})) \subset T_{\Omega}(p(\mbf{x}))$.
\end{proposition}

\begin{proposition}\label{prp:basis-normal-zero}
  Let $\Omega \subset \Real^m$ and  $\mbf{x} \in \Omega$.
  Further, let $\set{\mbf{v}_1,\dots,\mbf{v}_m}$ be a basis for $ \Real^m$.
  If $\pm \mbf{v}_i \in T_{\Omega}(\mbf{x})$ for all $i \in [m]$,
  then $N_{\Omega}(\mbf{x})=\set{\mbf{0}}$.
\end{proposition}

\begin{exercise}
  Prove \cref{prp:basis-normal-zero} using the definition of the normal cone.
\end{exercise}

\subsection{Best low-rank matrix approximation}
Amazingly, even though the set of matrices with rank less than or equal to $r$ is geometrically complicated, the low-rank matrix approximation has an optimal solution which is both easy to describe and efficient to compute.
The celebrated Eckart-Young Theorem states that the best rank-$r$ approximation is given by the rank-$r$ truncated SVD of the matrix whose approximation we seek \cite{eckart_approximation_1936}.

We demonstrate  the optimization tools introduced previously to prove a statement related to (but  weaker than) the Eckart-Young Theorem.
Our purpose is that the proof and its technique provide geometric intuition for the tensor results that follow.

\begin{proposition}
  \label{prp:matrix-no-nearest}
  Let $\A \in \Real^{m \times n}$ with $m \geq n$, and $\rank(\A) \geq r$.
  Then a solution to
  \begin{displaymath}
    \min_{ \rank(\B) \leq r} \frac{1}{2}\norm{\mathbf{A} - \mathbf{B}}_F^2
  \end{displaymath}
  exists and is of rank exactly $r$.
\end{proposition}
\begin{proof}
  Let $\Omega = \set{ \B \in \Real^{m \times n} | \rank(\B) \leq r }$.
  Since $\Omega$ is closed 
  (which we know from the Hasse diagram in \cref{fig:hasse-matrix}), 
  an optimal solution exists at some $\B^* \in \Omega$.

 Consider $\B \in \Omega$ with $\rank(\B) < r$.
  We will show $-\nabla f(\B) \notin N_{\Omega}(\B)$.  So $\B$ fails to satisfy the first-order optimal condition in \cref{thm:first-order-necessary}, and hence is not optimal.

  Let $\mbf{E}_{ij} \in \Real^{m \times n}$ be the rank-1 matrix with a 1 in position $(i,j)$ and zeros elsewhere.
  For all $\tau > 0$, we have $\B + \tau \mbf{E}_{ij} \in \Omega$
  since addition of a rank-1 matrix changes the rank by at most one. 
  From this, it is a small step to see $\Mx{E}{ij} \in T_{\Omega}(\B)$.
  Similar reasoning leads to $-\Mx{E}{ij} \in T_{\Omega}(\B)$.
  Since $\set{\mbf{E}_{ij}}$ forms a basis for $\Real^{m \times n}$ and the matrices $\set{ \pm \mbf{E}_{ij} } \subset T_{\Omega}(\B)$, we have $N_{\Omega}(\B) = \set{\mbf{0}}$ by \cref{prp:basis-normal-zero}.
On the other hand, we can check that
  $-\nabla f(\B) = \A - \B$.
  Since $\rank(\A) > \rank(\B)$, it follows
  $-\nabla f(\B) = \A - \B \neq \mbf{0}$.
Therefore $-\nabla f(\B) \notin N_{\Omega}(\B)$, and $\B$ is not a minimizer.

  Since $\B$ was an arbitrary point in $\Omega$ with rank less than $r$, we conclude that optimal solutions $\B^*$ (since they exist) must all have rank exactly $r$.
\end{proof}

A simple but subtle consequence of the proposition is that the low-rank matrix approximation problem over the set of \emph{exactly} rank-$r$ matrices has a solution for all target matrices $\A$ satisfying $\rank(\A) \geq r$, even though the feasible set is not closed.
In fact, matrices of rank strictly less than $r$ are in the \emph{singular locus} 
of the set $\set{ \A \in \Real^{m \times n} : \rank(\A) \leq r}$,
meaning that they occur at points of nonsmoothness \cite{harris1992algebraic}.
\Cref{fig:matrix-affine} provides an example: $\mbf{M}_0$ is a singular point of the curve where $\det(\mbf{M}) = 0$.
No vector in the visualized slice belongs to the normal cone of $\mbf{M}_0$ since every such vector has a positive inner product with at least one of the tangent directions.

\subsection{Best low-rank tensor approximation}
As we have said, the tensor low-rank approximation problem can be ill-posed.
The $\ttt$ setting provides a particularly pathological demonstration of ill-posedness: \emph{no} rank-3 tensor has a best rank-2 approximation.
We give a new direct proof of this statement.

\begin{theorem}[Ill-posedness of tensor approximation {\cite[Theorem 8.1]{de_silva_tensor_2008}}]
  \label{thm:ttt-no-nearest}
  No rank-3 tensor in $\Real^{\ttt}$ has a best rank-2 approximation.
\end{theorem}
\begin{proof}
  Let $\Omega \equiv D_0 \cup D_1 \cup D_2 \cup D_2^\prime \cup D_2^{\prime\prime} \cup G_2 \subset \Real^{\ttt}$ denote the set of tensors of rank at most 2.
  Let $\Tn{A} \in \Real^{\ttt}$ be such that $\rank(\Tn{A}) = 3$.
  We wish to show that
  $f(\Tn{B}) = \frac{1}{2}\norm{\Tn{A} - \Tn{B}}^2$
  does not attain a minimum in $\Omega$.

  If $\Tn{A} \in D_3$, then it has border rank 2 by \cref{thm:d3-brank}.
  Hence, it has no best rank-2 approximation.

  The other possibility is $\Tn{A} \in G_3$, where we can show something even stronger than the lack of a best approximant: there exist no first-order critical points.
  Taking the negative gradient and applying \cref{thm:first-order-necessary},
  any minimizer $\Tn{B}$ must satisfy
  $-\nabla f(\Tn{B}) = \Tn{A} - \Tn{B} \in N_{\Omega}(\Tn{B})$.
  By the assumption that $\rank(\Tn{A}) = 3$ and $\rank(\Tn{B}) \leq 2$, 
  we know $-\nabla f(\Tn{B}) \neq \mbf{0}$.
  Hence, to show the non-existence of first-order critical points, it suffices to prove
  $N_{\Omega}(\Tn{B}) = \set{ \mbf{0} }$. 
  We do this for each orbit with rank $\leq 2$.

\smallskip

  \textbf{Case 1:} Suppose $\Tn{B} \in D_0 \cup D_1$.
  We use an argument analogous to \cref{prp:matrix-no-nearest} but for tensors.
  If $\Tn{B} \in D_0 \cup D_1$, we can add any rank-1 tensor while keeping the rank $\leq 2$.
  Specifically, let
  $\Tn{E}{ijk}$ be the tensor of all zeros except a 1 in position $(i,j,k)$.
  We can argue that $\pm \Tn{E}{ijk} \in T_{\Omega}(\Tn{B})$ for all $i,j,k \in \set{1,2}$.
  Since $\set{\Tn{E}{ijk}}$ is a basis for $\Real^{\ttt}$, the normal cone 
  $N_{\Omega}(\Tn{B}) = \set{ \mbf{0} }$ by \cref{prp:basis-normal-zero}.

\smallskip

  \textbf{Case 2:} Suppose $\Tn{B} \in G_2$.
  By the arguments in \cref{sec:rank-2-3-typical}, we know that
  $G_2$ is an open set in $\Real^{\ttt}$.
  Therefore, the tangent cone at $\Tn{B}$ is $\Real^{\ttt}$
  and hence the normal cone is $\set{ \mbf{0} }$.

\smallskip

  \textbf{Case 3:} Suppose $\Tn{B} \in 
  D_2 \cup D_2^\prime \cup D_2^{\prime\prime}$.
  Since these three orbits are equivalent up to permutation of the modes, 
  we assume $\Tn{B} \in D_2^\prime$ without loss of generality.  
  We will show that $T_{\Omega}(\Tn{B})$ contains a basis of $\Real^{\ttt}$ and the negative of the basis, therefore $N_{\Omega}(\Tn{B}) = \set{\mbf{0}}$.
  
   Without loss of generality we can assume
  that $\Tn{B}$ is the canonical form of $D_2^\prime$.
  This is for the following reason. 
  Any tensor in $\Tn{C} \in D_2^\prime$ can be written as
  $\Tn{C} = \Tn{B} \times_1 \U \times_2 \V \times_3 \W$ for some
  $(\U,\V,\W) \in \GL(2) \times \GL(2) \times \GL(2)$.
  Since $\Omega$ is invariant under $\GL(2) \times \GL(2) \times \GL(2)$, it holds  $T_{\Omega}(\Tn{C}) = T_{\Omega}(\Tn{B}) \times_1 \U \times_2 \V \times_3 \W = \set{ \Tn{T} \times_1 \U \times_2 \V \times_3 \W : \Tn{T} \in T_{\Omega}(\Tn{B}) }$. 
  Since such a transformation by $(\U, \V, \W)$ is linear and invertible it takes bases to bases, hence if $T_{\Omega}(\Tn{B})$ contains a basis and its negative then so does 
  $T_{\Omega}(\Tn{C})$. 

  To obtain a subset of the tangent cone, we use \cref{prp:linear-subspace-tangent}.
  We let
  $f : (\Real^2)^{\otimes 6} \rightarrow \Real^{\ttt}$
  be defined by
  \begin{align*}
    f(\mbf{x}_1, \mbf{y}_1, \mbf{z}_1, \mbf{x}_2, \mbf{y}_2, \mbf{z}_2) =
    \mbf{x}_1 \otimes \mbf{y}_1 \otimes \mbf{z}_1 +
    \mbf{x}_2 \otimes \mbf{y}_2 \otimes \mbf{z}_2 ,
  \end{align*}
  so that the image of $f$ is  $\Omega$ and
  $f( \begin{bsmallmatrix} 1 \\ 0 \end{bsmallmatrix},
  \begin{bsmallmatrix} 1 \\ 0 \end{bsmallmatrix},
  \begin{bsmallmatrix} 1 \\ 0 \end{bsmallmatrix},
  \begin{bsmallmatrix} 1 \\ 0 \end{bsmallmatrix},
  \begin{bsmallmatrix} 0 \\ 1 \end{bsmallmatrix},
  \begin{bsmallmatrix} 0 \\ 1 \end{bsmallmatrix} ) = \Tn{B}$. 
  The column space of the Jacobian matrix of $f$ at this point,
  i.e.,
  $Df( \begin{bsmallmatrix} 1 \\ 0 \end{bsmallmatrix},
  \begin{bsmallmatrix} 1 \\ 0 \end{bsmallmatrix},
  \begin{bsmallmatrix} 1 \\ 0 \end{bsmallmatrix},
  \begin{bsmallmatrix} 1 \\ 0 \end{bsmallmatrix},
  \begin{bsmallmatrix} 0 \\ 1 \end{bsmallmatrix},
  \begin{bsmallmatrix} 0 \\ 1 \end{bsmallmatrix} )
  \in \Real^{8 \times 12}$,
  is spanned by vectorizations of the following 6 tensors:
  \begin{displaymath}
      \begin{tensoreight}[3ex]
        1 & 0 & 0 & 0 \\
        0 & 0 & 0 & 0
      \end{tensoreight},
      \begin{tensoreight}[3ex]
        0 & 0 & 0 & 1 \\
        0 & 0 & 0 & 0
      \end{tensoreight},
      \begin{tensoreight}[3ex]
        0 & 0 & 0 & 0 \\
        1 & 0 & 0 & 0
      \end{tensoreight},
      \begin{tensoreight}[3ex]
        0 & 0 & 0 & 0 \\
        0 & 0 & 0 & 1
      \end{tensoreight},
      \begin{tensoreight}[3ex]
        0 & 1 & 0 & 0 \\
        0 & 0 & 0 & 0
      \end{tensoreight},
      \begin{tensoreight}[3ex]
        0 & 0 & 1 & 0 \\
        0 & 0 & 0 & 0
      \end{tensoreight}.
  \end{displaymath}
  The Jacobian itself is given in \cref{sec:no-nearest-details}.
  By \cref{prp:linear-subspace-tangent}, the span of these vectors
  are in $T_{\Omega}(\Tn{B})$.
  The complement of the above span is given by the span of
  \begin{align}
      \Tn{C} =
      \begin{tensoreight}[3ex]
        0 & 0 & 0 & 0 \\
        0 & 1 & 0 & 0
      \end{tensoreight}
      \quad \text{and} \quad
      \Tn{D} =
      \begin{tensoreight}[3ex]
        0 & 0 & 0 & 0 \\
        0 & 0 & 1 & 0
      \end{tensoreight}.
  \end{align}
  A priori, these tensors might not be in the tangent cone.
  However, by inspection, both $\Tn{B} + \alpha \Tn{C}$ and $\Tn{B} + \alpha \Tn{D}$ lie in $D_3$ for any $\alpha \neq 0$.
  Therefore, $\pm \Tn{C}$ and $ \pm \Tn{D}$ are in the tangent cone at $\Tn{B}$.
  We can conclude that
  $\pm \Tn{E}{ijk} \in T_{\Omega}(\Tn{B})$ for all $i,j,k \in \set{1,2}$.
  Hence, by \cref{prp:basis-normal-zero},
  $N_{\Omega}(\Tn{B}) = \set{ \mbf{0} }$.   This finishes the casework.

\smallskip

    Since there are no rank-$2$ tensor satisfying the first-order critical condition when $\Tn{A} \in G_3$, there are no rank-2 minimizers.
\end{proof}

In particular, singularities of the feasible set give
rise to surprising behavior for the tensor low-rank approximation problem.
Consider \cref{fig:D2pp}, where the tensor $\Tn{T}_0 \in D_2^{\prime\prime}$ is of rank 2 and has arbitrarily close rank-3 neighbors in $G_3$.
According to \cref{thm:ttt-no-nearest}, $\Tn{T}_0$ is not the best rank-2 approximant to any of these tensors.
You can visually see how this happens:
for any fixed $G_3$ tensor $\Tn{A}$, there will always be a $G_2$ tensor that is closer to $\Tn{A}$ than $\Tn{T}_0$. A cartoon picture to have in mind is that any point in the interior of a triangle is always closer to an edge than to a corner.

One conclusion of \cite{de_silva_tensor_2008} is that the ill-posedness phenomenon is essentially independent of norms, since all are equivalent and induce the same topology.
Sequences converging in the Frobenius norm will converge in any other equivalent norm, so sets of bounded tensor rank remain non-closed.
We can add to this another observation: \cref{thm:ttt-no-nearest} holds independent of the norm, at least for smooth norms.  Indeed, $D_3$ tensors have convergent rank-2 sequences, and for $G_3$ tensors the gradient is still nonzero (by equivalence of norms)
while the normal cone is norm-independent (and so still equal to $\set{ \mbf{0} }$).

\subsection{Numerical implications}
What are the implications of ill-posedness of low-rank tensor approximation?
If a solution does not exist, what happens when a numerical optimization algorithm like alternating least squares \cite[Section 3.4]{kolda_tensor_2009}, nonlinear least squares \cite[Chapter 10]{nocedal_numerical_2006}, or even gradient descent is applied to the problem \eqref{eqn:tensor-low-rank}?
Iterative numerical algorithms produce a sequence of factor matrices $(\mbf{A}_k, \mbf{B}_k, \mbf{C}_k)_{k=1}^{\infty}$, and therefore a corresponding sequence of low-rank tensors $(\Tn[\hat]{T}{k})_{k=1}^{\infty}$ such that
$\Tn[\hat]{T}{k} = \llbracket \mbf{A}_k, \mbf{B}_k, \mbf{C}_k \rrbracket$.  
Let us suppose that the iterates successfully \emph{infimize} \eqref{eqn:tensor-low-rank}.  
Then it \emph{might} be the case that the tensors $(\Tn[\hat]{T}{k})_{k=1}^{\infty}$ converge, but it \emph{must} be the case that the factor matrices do not converge; specifically, at least two sequences of $(\mbf{A}_k)_{k=1}^{\infty}, (\mbf{B}_k)_{k=1}^{\infty}, (\mbf{C}_k)_{k=1}^{\infty}$ must diverge \cite[Proposition 4.8]{de_silva_tensor_2008}.  Further, if $(\Tn[\hat]{T}{k})_{k=1}^{\infty}$ has limit $\Tn[\hat]{T}$, then $\Tn[\hat]{T}$ has rank strictly greater than $r$ but border rank $\leq r$.  

To be concrete for the $\ttt$ case, if one attempts to compute a best rank-2 approximation to $\Tn{T} \in D_3$, then the iterates 
$\Tn[\hat]{T}{k} = \mbf{a}_{1,k} \circ \mbf{b}_{1,k} \circ \mbf{c}_{1,k} +
        \mbf{a}_{2,k} \circ \mbf{b}_{2,k} \circ \mbf{c}_{2,k}$
returned by a numerical algorithm may converge to the global infimum such that $\lim_{k \rightarrow \infty} \Tn[\hat]{T}{k} = \Tn{T}$, but
$\lim_{k \rightarrow \infty} \norm{\mbf{a}_{1,k} \circ \mbf{b}_{1,k} \circ \mbf{c}_{1,k}} = 
 \lim_{k \rightarrow \infty} \norm{\mbf{a}_{2,k} \circ \mbf{b}_{2,k} \circ \mbf{c}_{2,k}} =
 \infty$.
If instead $\Tn{T} \in G_3$ then $\Tn[\hat]{T}{k}$ could converge to
$\argmin_{\Tn[\hat]{T} : \text{border rank $\Tn[\hat]{T}$} \leq 2} \norm{\Tn{T} - \Tn[\hat]{T}}^2$
(which is in $D_3$) and the norms of each component would still diverge.

This situation is called ``diverging factors" in the literature.  
Ill-posedness can also lead to problems such as slowed convergence for numerical algorithms \cite{paatero_construction_2000, giordani_candecompparafac_2013}.
Interestingly, practitioners rarely observe diverging factors in practice for large tensors.
It is also possible to use safeguards against ill-posedness such as constraining the norms of the components (thereby restricting optimization to a compact set of factor matrices), constraining one factor matrix to be orthonormal \cite{krijnen_non-existence_2008}, or adding a regularization term to the cost function to penalize large-norm components \cite{AcDuKo11}.

\section{Beyond \texorpdfstring{$2 \times 2 \times 2$}{2 x 2 x 2} tensors}
\label{sec:beyond-ttt} 

Our examination of $\ttt$ tensors highlighted a variety of surprising phenomenological differences between matrices and tensors: multiple typical ranks, multiple orbits of the same rank, non-closedness of bounded rank sets, and ill-posedness of the low-rank approximation problem.
To obtain these results we used: groups and group actions, orbits, polynomials characterizing rank, cones and optimality, singularities, and algebraic geometry.

How many of these phenomena are particular to $2 \times 2 \times 2$ tensors?  Which features extend to tensors of bigger size?  

From \cref{sec:rank-semialgebraic}, tensors of bounded or fixed rank form semi-algebraic sets in general.
This implies the \emph{existence} of polynomial equations and inequalities that characterize the set (playing the roles of the hyperdeterminant and minors of the flattenings for $2 \times 2 \times 2$ tensors above).
However, explicit collections of such polynomials are not known for almost all larger tensor formats (nor are good degree bounds available to our knowledge).
In particular, the hyperdeterminant polynomial is not defined for many tensor formats \cite{gelfand_hyperdeterminants_1992}.
Even in cases where the hyperdeterminant is defined, it is generally not sufficient to characterize rank and border rank.
However, a semi-algebraic characterization for border rank-2 tensors in $\mathbb{R}^{m \times n \times p}$ is known, see \cite[Theorem 4.4]{seigal_real_2017}. This result is partly a consequence of the $2 \times 2 \times 2$ analysis applied to various $\ttt$ \emph{subtensors}.

Group actions are important for understanding tensor geometry in larger spaces, where the product group $\GL(m) \times \GL(n) \times \GL(p)$ acts on $\mathbb{R}^{m \times n \times p}$ through the multi-TTM operation.
However, the situation is murkier: there are usually infinitely many orbits.
In addition to matrices, covered above, the only tensor spaces with a finite number of orbits under the action of the respective product groups are $2 \times 2 \times a$ and $3 \times 2 \times a$ for $a \geq 1$ (\cite[Theorem 10.1.1.1]{landsberg_tensors_2012}, \cite{kac_remarks_1980}).
Thus, even tensors as small as $3 \times 3 \times 3$ will have an infinite number of orbits under the natural group action by invertible matrices.
This points to the need for additional tools to analyze larger tensor spaces.

Non-closed bounded-rank tensor sets still occur in higher dimensions, and so the headache of ill-posedness for low-rank tensor approximations also persists.
However, a quantitative understanding of ill-posedness is lacking.
For example, a positive volume of $n \times n \times n$ tensors have no best rank-2 approximation {\cite[Theorem 8.8]{de_silva_tensor_2008}}.
But it is open to describe the scaling of this volume, e.g., via lower/upper bounds on the probability of ill-posedness as a function of $n$.
In terms of positive results, given a tensor $\Tn{T} \in \Real^{n \times n \times n}$ (satisfying a variety of technical conditions) and an integer $r$, there exists an $\epsilon$-ball about $\Tn{T}$ such that all tensors in the ball have a best rank-$r$ approximation.
A lower bound on this radius $\epsilon$ has a deterministic formula and can be computed from entries of $\Tn{T}$ \cite{evert_guarantees_2022}.
Although computationally expensive, this approach gives a method to check well-posedness of a particular problem instance.

Multiple real typical ranks also occur in higher dimensions.
For example: $n \times n \times 2$ tensors have typical ranks $\set{n, n+1}$ \cite{ten_berge_simplicity_1999} and $12 \times 4 \times 4$ tensors have typical ranks $\set{12, 13}$ \cite{comon_generic_2009}.
Table 3.3 of \cite{kolda_tensor_2009} collects some general bounds for $m \times n \times p$ tensors where $m$, $n$, and $p$ satisfy various conditions.
Otherwise very limited information is available.

We also mention that although we restricted discussion to third-order tensors, all of the discussed properties and phenomena apply to higher-order tensors as well \cite{landsberg_tensors_2012}.

\section{Conclusion}
Tensors are a fertile area of research, unifying ideas and problems in statistics, numerical linear algebra, optimization, signal processing, algebraic geometry, and imaging, to name a few.  
Presently, tensors are seldom taught in the standard applied and computational mathematics courses.
Our examination of $\ttt$ tensors gives a peek into the complexity of tensors in general, especially phenomena related to rank.
By highlighting deviations from the familiar behavior of matrices we hope to have provided a glimpse of the rich, surprising, and often counterintuitive behavior of tensors.

\appendix

\section{Dimensions of tensor orbits}
\label{sec:tensor-orbit-details}
This section provides additional details regarding the proof of the tensor orbits dimensions in \cref{thm:tensor-orbit-dimension}.
First we provide the Jacobians for the other 7 orbits, and then we give an alternative derivation for these dimensions using an explicit Jacobian.

\subsection{Remaining Jacobians}
In \cref{thm:tensor-orbit-dimension} only the Jacobian of orbit $D_1$ was given.
The Jacobians for all 8 orbits are given in \cref{tbl:all-dimension-jacobians}, their ranks giving the claimed dimensions.
\begin{table}
  \begin{center}
    \begin{tblr}{c | c | c}
      orbit & Jacobian of $p$ at canonical form & dimension \\
      $D_0$ &
      $\begin{bsmallmatrix}
        0 & 0 & 0 & 0 & 0 & 0 & 0 & 0 & 0 & 0 & 0 & 0 \\
        0 & 0 & 0 & 0 & 0 & 0 & 0 & 0 & 0 & 0 & 0 & 0 \\
        0 & 0 & 0 & 0 & 0 & 0 & 0 & 0 & 0 & 0 & 0 & 0 \\
        0 & 0 & 0 & 0 & 0 & 0 & 0 & 0 & 0 & 0 & 0 & 0 \\
        0 & 0 & 0 & 0 & 0 & 0 & 0 & 0 & 0 & 0 & 0 & 0 \\
        0 & 0 & 0 & 0 & 0 & 0 & 0 & 0 & 0 & 0 & 0 & 0 \\
        0 & 0 & 0 & 0 & 0 & 0 & 0 & 0 & 0 & 0 & 0 & 0 \\
        0 & 0 & 0 & 0 & 0 & 0 & 0 & 0 & 0 & 0 & 0 & 0
      \end{bsmallmatrix}$ & 0 \\ \hline
      $D_1$ &
      $\begin{bsmallmatrix}
         1 & 0 & 0 & 0 & 1 & 0 & 0 & 0 & 1 & 0 & 0 & 0 \\
         0 & 1 & 0 & 0 & 0 & 0 & 0 & 0 & 0 & 0 & 0 & 0 \\
         0 & 0 & 0 & 0 & 0 & 1 & 0 & 0 & 0 & 0 & 0 & 0 \\
         0 & 0 & 0 & 0 & 0 & 0 & 0 & 0 & 0 & 0 & 0 & 0 \\
         0 & 0 & 0 & 0 & 0 & 0 & 0 & 0 & 0 & 1 & 0 & 0 \\
         0 & 0 & 0 & 0 & 0 & 0 & 0 & 0 & 0 & 0 & 0 & 0 \\
         0 & 0 & 0 & 0 & 0 & 0 & 0 & 0 & 0 & 0 & 0 & 0 \\
         0 & 0 & 0 & 0 & 0 & 0 & 0 & 0 & 0 & 0 & 0 & 0
      \end{bsmallmatrix}$ & 4 \\ \hline
      $D_2$ &
      $\begin{bsmallmatrix}
         1 & 0 & 0 & 0 & 1 & 0 & 0 & 0 & 1 & 0 & 0 & 0 \\
         0 & 1 & 0 & 0 & 0 & 0 & 1 & 0 & 0 & 0 & 0 & 0 \\
         0 & 0 & 1 & 0 & 0 & 1 & 0 & 0 & 0 & 0 & 0 & 0 \\
         0 & 0 & 0 & 1 & 0 & 0 & 0 & 1 & 1 & 0 & 0 & 0 \\
         0 & 0 & 0 & 0 & 0 & 0 & 0 & 0 & 0 & 1 & 0 & 0 \\
         0 & 0 & 0 & 0 & 0 & 0 & 0 & 0 & 0 & 0 & 0 & 0 \\
         0 & 0 & 0 & 0 & 0 & 0 & 0 & 0 & 0 & 0 & 0 & 0 \\
         0 & 0 & 0 & 0 & 0 & 0 & 0 & 0 & 0 & 1 & 0 & 0
      \end{bsmallmatrix}$ & 5 \\ \hline
      $D_2^\prime$ &
      $\begin{bsmallmatrix}
         1 & 0 & 0 & 0 & 1 & 0 & 0 & 0 & 1 & 0 & 0 & 0 \\
         0 & 1 & 0 & 0 & 0 & 0 & 0 & 0 & 0 & 0 & 0 & 0 \\
         0 & 0 & 0 & 0 & 0 & 1 & 0 & 0 & 0 & 0 & 1 & 0 \\
         0 & 0 & 0 & 0 & 0 & 0 & 0 & 0 & 0 & 0 & 0 & 0 \\
         0 & 0 & 0 & 0 & 0 & 0 & 1 & 0 & 0 & 1 & 0 & 0 \\
         0 & 0 & 0 & 0 & 0 & 0 & 0 & 0 & 0 & 0 & 0 & 0 \\
         1 & 0 & 0 & 0 & 0 & 0 & 0 & 1 & 0 & 0 & 0 & 1 \\
         0 & 1 & 0 & 0 & 0 & 0 & 0 & 0 & 0 & 0 & 0 & 0
      \end{bsmallmatrix}$ & 5 \\ \hline
      $D_2^{\prime\prime}$ &
      $\begin{bsmallmatrix}
         1 & 0 & 0 & 0 & 1 & 0 & 0 & 0 & 1 & 0 & 0 & 0 \\
         0 & 1 & 0 & 0 & 0 & 0 & 0 & 0 & 0 & 0 & 1 & 0 \\
         0 & 0 & 0 & 0 & 0 & 1 & 0 & 0 & 0 & 0 & 0 & 0 \\
         0 & 0 & 0 & 0 & 0 & 0 & 0 & 0 & 0 & 0 & 0 & 0 \\
         0 & 0 & 1 & 0 & 0 & 0 & 0 & 0 & 0 & 1 & 0 & 0 \\
         0 & 0 & 0 & 1 & 1 & 0 & 0 & 0 & 0 & 0 & 0 & 1 \\
         0 & 0 & 0 & 0 & 0 & 0 & 0 & 0 & 0 & 0 & 0 & 0 \\
         0 & 0 & 0 & 0 & 0 & 1 & 0 & 0 & 0 & 0 & 0 & 0
      \end{bsmallmatrix}$ & 5 \\ \hline
      $D_3$ &
      $\begin{bsmallmatrix}
         1 & 0 & 0 & 0 & 1 & 0 & 0 & 0 & 1 & 0 & 0 & 0 \\
         0 & 1 & 0 & 0 & 0 & 0 & 0 & 0 & 0 & 0 & 1 & 0 \\
         0 & 0 & 0 & 0 & 0 & 1 & 0 & 0 & 0 & 0 & 1 & 0 \\
         0 & 0 & 0 & 0 & 0 & 0 & 0 & 0 & 0 & 0 & 0 & 0 \\
         0 & 0 & 1 & 0 & 0 & 0 & 1 & 0 & 0 & 1 & 0 & 0 \\
         0 & 0 & 0 & 1 & 1 & 0 & 0 & 0 & 0 & 0 & 0 & 1 \\
         1 & 0 & 0 & 0 & 0 & 0 & 0 & 1 & 0 & 0 & 0 & 1 \\
         0 & 1 & 0 & 0 & 0 & 1 & 0 & 0 & 0 & 0 & 0 & 0
      \end{bsmallmatrix}$ & 7 \\ \hline
      $G_2$ &
      $\begin{bsmallmatrix}
         1 & 0 & 0 & 0 & 1 & 0 & 0 & 0 & 1 & 0 & 0 & 0 \\
         0 & 1 & 0 & 0 & 0 & 0 & 0 & 0 & 0 & 0 & 0 & 0 \\
         0 & 0 & 0 & 0 & 0 & 1 & 0 & 0 & 0 & 0 & 0 & 0 \\
         0 & 0 & 0 & 0 & 0 & 0 & 0 & 0 & 0 & 0 & 1 & 0 \\
         0 & 0 & 0 & 0 & 0 & 0 & 0 & 0 & 0 & 1 & 0 & 0 \\
         0 & 0 & 0 & 0 & 0 & 0 & 1 & 0 & 0 & 0 & 0 & 0 \\
         0 & 0 & 1 & 0 & 0 & 0 & 0 & 0 & 0 & 0 & 0 & 0 \\
         0 & 0 & 0 & 1 & 0 & 0 & 0 & 1 & 0 & 0 & 0 & 1
      \end{bsmallmatrix}$ & 8 \\ \hline
      $G_3$ &
      $\begin{bsmallmatrix*}[r]
         \phantom{-}1 &  \phantom{-}0 &  \phantom{-}0 &  \phantom{-}0 &  \phantom{-}1 &  \phantom{-}0 &  \phantom{-}0 &  \phantom{-}0 &  \phantom{-}1 &  \phantom{-}0 &  \phantom{-}0 &  \phantom{-}0 \\
         0 &  1 &  0 &  0 &  0 &  0 &  1 &  0 &  0 &  0 &  1 &  0 \\
         0 &  0 &  1 &  0 &  0 &  1 &  0 &  0 &  0 &  0 & -1 &  0 \\
         0 &  0 &  0 &  1 &  0 &  0 &  0 &  1 &  1 &  0 &  0 &  0 \\
         0 &  0 &  1 &  0 &  0 &  0 & -1 &  0 &  0 &  1 &  0 &  0 \\
         0 &  0 &  0 &  1 &  1 &  0 &  0 &  0 &  0 &  0 &  0 &  1 \\
        -1 &  0 &  0 &  0 &  0 &  0 &  0 & -1 &  0 &  0 &  0 & -1 \\
         0 & -1 &  0 &  0 &  0 &  1 &  0 &  0 &  0 &  1 &  0 &  0
      \end{bsmallmatrix*}$ & 8 \\ \hline
    \end{tblr}
  \end{center}
  \caption{Jacobians of the multi-TTM operation on the respective canonical forms of each orbit.
    The rank of each Jacobian gives the dimension of the respective orbit.
    The 8 rows correspond to
    $( f_{111}, f_{211}, f_{121}, f_{221}, f_{112}, f_{212}, f_{122}, f_{222} )$
    and the 12 columns correspond to
    $( u_{11}, u_{21}, u_{12}, u_{22},
       v_{11}, v_{21}, v_{12}, v_{22},
       w_{11}, w_{21}, w_{12}, w_{22} )$.
  }
  \label{tbl:all-dimension-jacobians}
\end{table}

\subsection{Alternate proof using matrix calculus}
We can give an explicit expression for the Jacobian of the tensor map
\begin{displaymath}
  f(\U,\V,\W) = \G \times_1 \U \times_2 \V \times_3 \W,
\end{displaymath}
where we treat the input and output as vectors in $\Real^{12}$ and $\Real^8$, respectively.
The Jacobian involves Kronecker products, denoted $\otimes$,
and the tensor perfect shuffle matrices \cite{kolda_tensor_2025}.

\begin{proposition}[Explicit Jacobian of Tensor Map]
  \label{prop:jacobian-matrix-explicit}
  Define the map $f: \Real^{12} \rightarrow \Real^8$ by
  \begin{displaymath}
    f\prn[\Bigg]{
      \begin{bmatrix}
        \vec(\U) \\ \vec(\V) \\ \vec(\W)
      \end{bmatrix}
    }
    =
    \vec\prn[\Big]{
      \G \times_1 \U \times_2 \V \times_3 \W
    }
  \end{displaymath}
  where $\G$ is a $\ttt$ tensor.
  Then the Jacobian of $f$ is the $8 \times 12$ matrix given by
  \begin{displaymath}
    \Mx{J} =
    \text{\small$
        \begin{bmatrix}
          (\W \otimes \V)^{\Tr} \Tm{G}{1}' \otimes \Mx{I}{2}                           &
          \Mx{P}{2}' \left( (\W \otimes \U)^{\Tr} \Tm{G}{2}' \otimes \Mx{I}{2} \right) &
          \Mx{P}{3}' \left( (\V \otimes \U)^{\Tr} \Tm{G}{3}' \otimes \Mx{I}{2} \right)
        \end{bmatrix}
      $}.
  \end{displaymath}
  Here, $\Mx{P}{2}$ and $\Mx{P}{3}$ are tensor perfect shuffle matrices such that
  $\Mx{P}{2} \vec(\T) = \vec(\Tm{T}{2})$ and
  $\Mx{P}{3} \vec(\T) = \vec(\Tm{T}{3})$.
\end{proposition}

\begin{proof}
  The function $f$ is a Tucker tensor. 
  We can use expressions for the mode-unfoldings of a Tucker tensor
  \cite[Proposition 5.5]{kolda_tensor_2025}
  and the Kronecker product relation that says 
  $\vec(\A\B\C) = (\C' \otimes \A)\vec(\B)$
  to express $f$ equivalently as 
  \begin{align*}
    f
     & = \vec\prn[\Big]{
      \U \Tm{G}{1} (\W \otimes \V)
    }
    = \prn[\Big]{(\W \otimes \V)^{\Tr} \Tm{G}{1}' \otimes \Mx{I}{2}  } \vec(\U)
    \\
     & = \Mx{P}{2}' \vec\prn[\Big]{
      \V \Tm{G}{2} (\W \otimes \U)
    }
    = \Mx{P}{2}' 
    \prn[\Big]{(\W \otimes \U)^{\Tr} \Tm{G}{2}' \otimes \Mx{I}{2}  } \vec(\V)
    \\
     & = \Mx{P}{3}' \vec\prn[\Big]{
      \W \Tm{G}{3} (\V \otimes \U)
    }
    = \Mx{P}{3}' 
    \prn[\Big]{(\V \otimes \U)^{\Tr} \Tm{G}{3}' \otimes \Mx{I}{2}  } \vec(\W).
  \end{align*}
  The result follows from the fact that the Jacobian is the matrix of partial derivatives and the derivative of $\A\Vc{X}$ is $\A$.
\end{proof}

Along with this different derivation, we provide a direct proof
that the rank of the Jacobian is independent of the choice of 
$\U$, $\V$, and $\W$. This relies heavily on the properties of the Kronecker 
product and the tensor perfect shuffle matrix.

\begin{proposition}[Independence of Jacobian Rank]
  \label{prop:jacobian-rank-independence}
  Let $\U,\V,\W \in \GL(2)$ and define the nonsingular matrices
  \begin{align*}
    \Mx{M}{1} &\equiv (\W^{-1} \otimes \V^{-1})^{\Tr} \otimes \U^{-1} 
    \in \Real^{8 \times 8} \quad\text{and}\\
    \Mx{M}{2} &\equiv \text{\rm blkdiag}(\Mx{I}{2} \otimes \Mx{U}, \Mx{I}{2} \otimes \Mx{V}, \Mx{I}{2} \otimes \Mx{W}) 
    \in \Real^{12 \times 12}.
  \end{align*}
  Then the Jacobian $\Mx{J}$ defined in \cref{prop:jacobian-matrix-explicit} 
  has the same rank as
  \begin{displaymath}
    \Mx[\tilde]{J} = \Mx{M}{1} \Mx{J} \Mx{M}{2} =
    \begin{bmatrix}
      \Tm{G}{1}' \otimes \Mx{I}{2}                           &
      \Mx{P}{2}' \left( \Tm{G}{2}' \otimes \Mx{I}{2} \right) &
      \Mx{P}{3}' \left( \Tm{G}{3}' \otimes \Mx{I}{2} \right)
    \end{bmatrix}.
  \end{displaymath}
\end{proposition}

\begin{proof}
  Labeling the three blocks of $\Mx{J}$ as $\Mx{J}{1}$, $\Mx{J}{2}$, and 
  $\Mx{J}{3}$ and using properties of Kronecker products and matrix 
  multiplication, we. have
  \begin{displaymath}
    \Mx[\tilde]{J} = \Mx{M}{1} \Mx{J} \Mx{M}{2} 
    = \begin{bmatrix}
      \Mx{M}{1} \Mx{J}{1} (\Mx{I}{2} \otimes \U) &
      \Mx{M}{1} \Mx{J}{2} (\Mx{I}{2} \otimes \V) &
      \Mx{M}{1} \Mx{J}{3} (\Mx{I}{2} \otimes \W)
    \end{bmatrix}.
  \end{displaymath}
  Its first block is
  \begin{displaymath}
    \Mx[\tilde]{J}{1} = 
    [(\W^{-1} \otimes \V^{-1})^{\Tr} \otimes \U^{-1}] 
    [(\W \otimes \V)^{\Tr} \Tm{G}{1}' \otimes \Mx{I}{2}]
     (\Mx{I}{2} \otimes \U) = \Tm{G}{1}' \otimes \Mx{I}{2}.
  \end{displaymath}
  The second and third blocks are more complex because
  they involve the tensor perfect shuffle matrix.
  However, we can complete those by observing that we have
  \begin{align*}
    \Mx{M} \Mx{P}{2}' & = \Mx{P}{2}' \left[ (\W^{-1} \otimes \U^{-1})^{\Tr} \otimes \V^{-1} \right], \quad \text{and} \\
    \Mx{M} \Mx{P}{3}' & = \Mx{P}{3}' \left[ (\V^{-1} \otimes \U^{-1})^{\Tr} \otimes \W^{-1} \right].
  \end{align*}
  Hence, the claim.
\end{proof}

Now that we have established these tools, all that is left is to compute 
the rank of the transformed Jacobian for each canonical form.
For our computations, the perfect shuffle matrices are
\begin{displaymath}
  \Mx{P}{2} =
  \begin{bmatrix}
    1 & 0 & 0 & 0 & 0 & 0 & 0 & 0 \\
    0 & 0 & 1 & 0 & 0 & 0 & 0 & 0 \\
    0 & 1 & 0 & 0 & 0 & 0 & 0 & 0 \\
    0 & 0 & 0 & 1 & 0 & 0 & 0 & 0 \\
    0 & 0 & 0 & 0 & 1 & 0 & 0 & 0 \\
    0 & 0 & 0 & 0 & 0 & 0 & 1 & 0 \\
    0 & 0 & 0 & 0 & 0 & 1 & 0 & 0 \\
    0 & 0 & 0 & 0 & 0 & 0 & 0 & 1
  \end{bmatrix}
  \quad\text{and}\quad
  \Mx{P}{3} =
  \begin{bmatrix}
    1 & 0 & 0 & 0 & 0 & 0 & 0 & 0 \\
    0 & 0 & 0 & 0 & 1 & 0 & 0 & 0 \\
    0 & 1 & 0 & 0 & 0 & 0 & 0 & 0 \\
    0 & 0 & 0 & 0 & 0 & 1 & 0 & 0 \\
    0 & 0 & 1 & 0 & 0 & 0 & 0 & 0 \\
    0 & 0 & 0 & 0 & 0 & 0 & 1 & 0 \\
    0 & 0 & 0 & 1 & 0 & 0 & 0 & 0 \\
    0 & 0 & 0 & 0 & 0 & 0 & 0 & 1
  \end{bmatrix}.
\end{displaymath}

Jacobians are presented with rows and columns corresponding to the input and output of $f$ in an identical fashion to \cref{tbl:all-dimension-jacobians}.

\textbf{Case I: $D_0$.} In this case, the canonical form $\G$ is the
all-zero tensor and so $\Mx[\tilde]{J}$ is an all-zero matrix with rank 0.

\textbf{Case II: $D_1$.} In this case, the canonical $\G$ has unfoldings
\begin{displaymath}
  \Tm{G}{1}' = \Tm{G}{2}' = \Tm{G}{3}' =
  \begin{bmatrix}
    1 & 0 \\ 0 & 0 \\ 0 & 0 \\ 0 & 0
  \end{bmatrix},
\end{displaymath}
so the transformed Jacobian is
\begin{displaymath}
  \Mx[\tilde]{J} =
  \begin{bmatrix}
    1 & 0 & 0 & 0 & 1 & 0 & 0 & 0 & 1 & 0 & 0 & 0 \\
    0 & 1 & 0 & 0 & 0 & 0 & 0 & 0 & 0 & 0 & 0 & 0 \\
    0 & 0 & 0 & 0 & 0 & 1 & 0 & 0 & 0 & 0 & 0 & 0 \\
    0 & 0 & 0 & 0 & 0 & 0 & 0 & 0 & 0 & 0 & 0 & 0 \\
    0 & 0 & 0 & 0 & 0 & 0 & 0 & 0 & 0 & 1 & 0 & 0 \\
    0 & 0 & 0 & 0 & 0 & 0 & 0 & 0 & 0 & 0 & 0 & 0 \\
    0 & 0 & 0 & 0 & 0 & 0 & 0 & 0 & 0 & 0 & 0 & 0 \\
    0 & 0 & 0 & 0 & 0 & 0 & 0 & 0 & 0 & 0 & 0 & 0
  \end{bmatrix},
\end{displaymath}
which has rank 4.

\textbf{Case III: $D_2$, $D_2'$, $D_2''$.}
We can group these three cases together because their canonical forms are permutation of
one another and permutation does not change rank.
Considering the canonical form $\G$ of $D_2$, its unfoldings are
\begin{displaymath}
  \Tm{G}{1}' = \Tm{G}{2}' =
  \begin{bmatrix}
    1 & 0 \\ 0 & 1 \\ 0 & 0 \\ 0 & 0
  \end{bmatrix}
  \quad\text{and}\quad
  \Tm{G}{3}' =
  \begin{bmatrix}
    1 & 0 \\ 0 & 0 \\ 0 & 0 \\ 1 & 0
  \end{bmatrix}.
\end{displaymath}
Then the transformed Jacobian is
\begin{displaymath}
  \Mx[\tilde]{J}=
  \begin{bmatrix}
    1 & 0 & 0 & 0 & 1 & 0 & 0 & 0 & 1 & 0 & 0 & 0 \\
    0 & 1 & 0 & 0 & 0 & 0 & 1 & 0 & 0 & 0 & 0 & 0 \\
    0 & 0 & 1 & 0 & 0 & 1 & 0 & 0 & 0 & 0 & 0 & 0 \\
    0 & 0 & 0 & 1 & 0 & 0 & 0 & 1 & 1 & 0 & 0 & 0 \\
    0 & 0 & 0 & 0 & 0 & 0 & 0 & 0 & 0 & 1 & 0 & 0 \\
    0 & 0 & 0 & 0 & 0 & 0 & 0 & 0 & 0 & 0 & 0 & 0 \\
    0 & 0 & 0 & 0 & 0 & 0 & 0 & 0 & 0 & 0 & 0 & 0 \\
    0 & 0 & 0 & 0 & 0 & 0 & 0 & 0 & 0 & 1 & 0 & 0
  \end{bmatrix},
\end{displaymath}
which has rank 5.

\textbf{Case IV: $D_3$.}
The unfoldings of the canonical form $\G$ are:
\begin{displaymath}
  \Tm{G}{1}' = \Tm{G}{2}' =
  \begin{bmatrix}
    1 & 0 \\
    0 & 0 \\
    0 & 1 \\
    1 & 0
  \end{bmatrix}
  \quad\text{and}\quad
  \Tm{G}{3}' =
  \begin{bmatrix}
    1 & 0 \\
    0 & 1 \\
    0 & 1 \\
    0 & 0
  \end{bmatrix}.
\end{displaymath}
The transformed Jacobian is
\begin{displaymath}
  \Mx[\tilde]{J} =
  \begin{bmatrix}
    1 & 0 & 0 & 0 & 1 & 0 & 0 & 0 & 1 & 0 & 0 & 0 \\
    0 & 1 & 0 & 0 & 0 & 0 & 0 & 0 & 0 & 0 & 1 & 0 \\
    0 & 0 & 0 & 0 & 0 & 1 & 0 & 0 & 0 & 0 & 1 & 0 \\
    0 & 0 & 0 & 0 & 0 & 0 & 0 & 0 & 0 & 0 & 0 & 0 \\
    0 & 0 & 1 & 0 & 0 & 0 & 1 & 0 & 0 & 1 & 0 & 0 \\
    0 & 0 & 0 & 1 & 1 & 0 & 0 & 0 & 0 & 0 & 0 & 1 \\
    1 & 0 & 0 & 0 & 0 & 0 & 0 & 1 & 0 & 0 & 0 & 1 \\
    0 & 1 & 0 & 0 & 0 & 1 & 0 & 0 & 0 & 0 & 0 & 0
  \end{bmatrix},
\end{displaymath}
which has rank 7.

\textbf{Case V: $G_2$.}
The canonical form $\G$ has unfoldings
\begin{displaymath}
  \Tm{G}{1}' = \Tm{G}{2}' = \Tm{G}{3}' =
  \begin{bmatrix}
    1 & 0 \\ 0 & 0 \\ 0 & 0 \\ 0 & 1
  \end{bmatrix}.
\end{displaymath}
The transformed Jacobian is
\begin{displaymath}
  \Mx[\tilde]{J} =
  \begin{bmatrix}
    1 & 0 & 0 & 0 & 1 & 0 & 0 & 0 & 1 & 0 & 0 & 0 \\
    0 & 1 & 0 & 0 & 0 & 0 & 0 & 0 & 0 & 0 & 0 & 0 \\
    0 & 0 & 0 & 0 & 0 & 1 & 0 & 0 & 0 & 0 & 0 & 0 \\
    0 & 0 & 0 & 0 & 0 & 0 & 0 & 0 & 0 & 0 & 1 & 0 \\
    0 & 0 & 0 & 0 & 0 & 0 & 0 & 0 & 0 & 1 & 0 & 0 \\
    0 & 0 & 0 & 0 & 0 & 0 & 1 & 0 & 0 & 0 & 0 & 0 \\
    0 & 0 & 1 & 0 & 0 & 0 & 0 & 0 & 0 & 0 & 0 & 0 \\
    0 & 0 & 0 & 1 & 0 & 0 & 0 & 1 & 0 & 0 & 0 & 1
  \end{bmatrix},
\end{displaymath}
which has rank 8.

\textbf{Case VI: $G_3$.}
In this case, the canonical $\G$ had unfoldings
\begin{displaymath}
  \Tm{G}{1}' =
  \begin{bmatrix*}[r]
    1 & 0 \\
    0 & 1 \\
    0 & 1 \\
    -1 & 0
  \end{bmatrix*}
  \quad\text{and}\quad
  \Tm{G}{2}'= \Tm{G}{3}' =
  \begin{bmatrix*}[r]
    1 & 0 \\
    0 & 1 \\
    0 & -1 \\
    1 & 0
  \end{bmatrix*} .
\end{displaymath}
The transformed Jacobian is
\begin{displaymath}
  \Mx[\tilde]{J} =
  \begin{bmatrix*}[r]
    \phantom{-}1 &  \phantom{-}0 &  \phantom{-}0 &  \phantom{-}0 &  \phantom{-}1 &  \phantom{-}0 &  \phantom{-}0 &  \phantom{-}0 &  \phantom{-}1 &  \phantom{-}0 &  \phantom{-}0 &  \phantom{-}0 \\
    0 &  1 &  0 &  0 &  0 &  0 &  1 &  0 &  0 &  0 &  1 &  0 \\
    0 &  0 &  1 &  0 &  0 &  1 &  0 &  0 &  0 &  0 & -1 &  0 \\
    0 &  0 &  0 &  1 &  0 &  0 &  0 &  1 &  1 &  0 &  0 &  0 \\
    0 &  0 &  1 &  0 &  0 &  0 & -1 &  0 &  0 &  1 &  0 &  0 \\
    0 &  0 &  0 &  1 &  1 &  0 &  0 &  0 &  0 &  0 &  0 &  1 \\
   -1 &  0 &  0 &  0 &  0 &  0 &  0 & -1 &  0 &  0 &  0 & -1 \\
    0 & -1 &  0 &  0 &  0 &  1 &  0 &  0 &  0 &  1 &  0 &  0
  \end{bmatrix*},
\end{displaymath}
which has rank 8.

\section{Additional details on ill-posedness of rank-2 approximation for $\ttt$ tensors}
\label{sec:no-nearest-details}
The proof of \cref{thm:ttt-no-nearest} uses the Jacobian of the rank-2 parameterization for $\ttt$ tensors.
The full Jacobian at parameters mapping to the canonical form of $D_2'$ is
\begin{displaymath}
  \begin{bmatrix}
    1 & 0 & 0 & 0 & 1 & 0 & 0 & 0 & 1 & 0 & 0 & 0 \\
    0 & 1 & 0 & 0 & 0 & 0 & 0 & 0 & 0 & 0 & 0 & 0 \\
    0 & 0 & 0 & 0 & 0 & 1 & 0 & 0 & 0 & 0 & 1 & 0 \\
    0 & 0 & 0 & 0 & 0 & 0 & 0 & 0 & 0 & 0 & 0 & 0 \\
    0 & 0 & 0 & 0 & 0 & 0 & 1 & 0 & 0 & 1 & 0 & 0 \\
    0 & 0 & 0 & 0 & 0 & 0 & 0 & 0 & 0 & 0 & 0 & 0 \\
    0 & 0 & 1 & 0 & 0 & 0 & 0 & 1 & 0 & 0 & 0 & 1 \\
    0 & 0 & 0 & 1 & 0 & 0 & 0 & 0 & 0 & 0 & 0 & 0
  \end{bmatrix},
\end{displaymath}
with the row and column ordering again matching that in \cref{tbl:all-dimension-jacobians}.
Bases for the column space and its complement are immediate upon inspection.
As noted in the original proof, the basis for the column space is guaranteed to be in the tangent space; meanwhile a simple calculation shows the elementary basis for the complement also happens to be in the tangent space.
Reshaping these 8 basis vectors into $\ttt$ tensors gives the $\Tn{E}_{ijk}$ tensors.

\bibliographystyle{neutral}
\bibliography{ref.bib}

\begin{thebibliography}{10}

\bibitem{AcDuKo11}
E.~Acar, D.~M. Dunlavy, and T.~G. Kolda, {\em A {S}calable {O}ptimization
  {A}pproach for {F}itting {C}anonical {T}ensor {D}ecompositions}, Journal of
  Chemometrics, 25 (2011), pp.~67--86, \url{https://doi.org/10.1002/cem.1335}.

\bibitem{basu_algorithms_2006}
S.~Basu, R.~Pollack, and M.-F. Roy, {\em Algorithms in Real Algebraic
  Geometry}, no.~v. 10 in Algorithms and Computation in Mathematics, Springer,
  Berlin ; New York, 2nd~ed., 2006.

\bibitem{bergqvist_exact_2013}
G.~Bergqvist, {\em Exact {P}robabilities for {T}ypical {R}anks of
  2{\texttimes}2{\texttimes}2 and 3{\texttimes}3{\texttimes}2 {T}ensors},
  Linear Algebra and its Applications, 438 (2013), pp.~663--667,
  \url{https://doi.org/10.1016/j.laa.2011.02.041}.

\bibitem{bini_approximate_1980}
D.~Bini, G.~Lotti, and F.~Romani, {\em Approximate {{Solutions}} for the
  {{Bilinear Form Computational Problem}}}, SIAM Journal on Computing, 9
  (1980), pp.~692--697, \url{https://doi.org/10.1137/0209053}.

\bibitem{bini_role_2007}
D.~A. Bini, {\em The {R}ole of {T}ensor {R}ank in the {C}omplexity {A}nalysis
  of {B}ilinear {F}orms}, 2007.

\bibitem{birkhoff_lattice_1967}
G.~Birkhoff, {\em Lattice Theory}, no.~25 in Colloquium {{Publications}},
  American Mathematical Society, Providence, [3d. ed.] (online-ausg.)~ed.,
  1967.

\bibitem{bochnak_real_1998}
J.~Bochnak, M.~Coste, and M.-F. Roy, {\em Real Algebraic Geometry}, Springer
  Berlin Heidelberg, Berlin, Heidelberg, 1998,
  \url{https://doi.org/10.1007/978-3-662-03718-8}.

\bibitem{boumal2023introduction}
N.~Boumal, {\em An {I}ntroduction to {O}ptimization on {S}mooth {M}anifolds},
  Cambridge University Press, 2023.

\bibitem{cayley_theory_1845}
A.~Cayley, {\em On the {T}heory of {L}inear {T}ransformations}, Cambridge
  Journal of Mathematics, 4 (1845), pp.~193--209.

\bibitem{cayley_collected_2009}
A.~Cayley, {\em The Collected Mathematical Papers}, Cambridge University Press,
  1~ed., July 2009, \url{https://doi.org/10.1017/CBO9780511703676}.

\bibitem{comon_generic_2009}
P.~Comon, J.~Ten~Berge, L.~De~Lathauwer, and J.~Castaing, {\em Generic and
  {T}ypical {R}anks of {M}ulti-{W}ay {A}rrays}, Linear Algebra and its
  Applications, 430 (2009), pp.~2997--3007,
  \url{https://doi.org/10.1016/j.laa.2009.01.014}.

\bibitem{de_silva_tensor_2008}
V.~{de Silva} and L.-H. Lim, {\em Tensor {R}ank and the {I}ll-{P}osedness of
  the {B}est {L}ow-{R}ank {A}pproximation {P}roblem}, SIAM Journal on Matrix
  Analysis and Applications, 30 (2008), pp.~1084--1127,
  \url{https://doi.org/10.1137/06066518X}.

\bibitem{dolgachev2012classical}
I.~V. Dolgachev, {\em Classical {A}lgebraic {G}eometry: {A} {M}odern {V}iew},
  Cambridge University Press, 2012.

\bibitem{eckart_approximation_1936}
C.~Eckart and G.~Young, {\em The {A}pproximation of {O}ne {M}atrix by {A}nother
  of {L}ower {R}ank}, Psychometrika, 1 (1936), pp.~211--218,
  \url{https://doi.org/10.1007/BF02288367}.

\bibitem{evert_guarantees_2022}
E.~Evert and L.~De~Lathauwer, {\em Guarantees for {E}xistence of a {B}est
  {C}anonical {P}olyadic {A}pproximation of a {N}oisy {L}ow-{R}ank {T}ensor},
  SIAM Journal on Matrix Analysis and Applications, 43 (2022), pp.~328--369,
  \url{https://doi.org/10.1137/20M1381046}.

\bibitem{gelfand_hyperdeterminants_1992}
I.~Gelfand, M.~Kapranov, and A.~Zelevinsky, {\em Hyperdeterminants}, Advances
  in Mathematics, 96 (1992), pp.~226--263,
  \url{https://doi.org/10.1016/0001-8708(92)90056-Q}.

\bibitem{giordani_candecompparafac_2013}
P.~Giordani and R.~Rocci, {\em {{CANDECOMP}}/{{PARAFAC}} with {R}idge
  {R}egularization}, Chemometrics and Intelligent Laboratory Systems, 129
  (2013), pp.~3--9, \url{https://doi.org/10.1016/j.chemolab.2013.08.002}.

\bibitem{GoVa96}
G.~H. Golub and C.~F. {Van Loan}, {\em Matrix Computations}, Johns Hopkins
  Univ. Press, 1996.

\bibitem{harris1992algebraic}
J.~Harris, {\em Algebraic {G}eometry: {A} {F}irst {C}ourse}, vol.~133, Springer
  Science \& Business Media, 1992.

\bibitem{higham_stability_1992}
N.~J. Higham, {\em Stability of a {{Method}} for {{Multiplying Complex
  Matrices}} with {{Three Real Matrix Multiplications}}}, SIAM Journal on
  Matrix Analysis and Applications, 13 (1992), pp.~681--687,
  \url{https://doi.org/10.1137/0613043}.

\bibitem{hillar_most_2013}
C.~J. Hillar and L.-H. Lim, {\em Most {T}ensor {P}roblems {A}re {NP}-{H}ard},
  Journal of the ACM, 60 (2013), pp.~45:1--45:39,
  \url{https://doi.org/10.1145/2512329}.

\bibitem{horn_matrix_2017}
R.~A. Horn and C.~R. Johnson, {\em Matrix Analysis}, Cambridge University
  Press, New York, NY, second edition, corrected reprint~ed., 2017.

\bibitem{kac_remarks_1980}
V.~G. Kac, {\em Some {R}emarks on {N}ilpotent {O}rbits}, Journal of Algebra, 64
  (1980), pp.~190--213, \url{https://doi.org/10.1016/0021-8693(80)90141-6}.

\bibitem{kolda_tensor_2009}
T.~G. Kolda and B.~W. Bader, {\em Tensor decompositions and applications}, SIAM
  Review, 51 (2009), pp.~455--500, \url{https://doi.org/10.1137/07070111X}.

\bibitem{kolda_tensor_2025}
T.~G. Kolda and G.~Ballard, {\em Tensor Decompositions for Data Science},
  Cambridge University Press, 2025.

\bibitem{krijnen_non-existence_2008}
W.~P. Krijnen, T.~K. Dijkstra, and A.~Stegeman, {\em On the {{Non-Existence}}
  of {{Optimal Solutions}} and the {{Occurrence}} of ``{{Degeneracy}}'' in the
  {{CANDECOMP}}/{{PARAFAC Model}}}, Psychometrika, 73 (2008), pp.~431--439,
  \url{https://doi.org/10.1007/s11336-008-9056-1}.

\bibitem{kruskal_rank_1983}
J.~Kruskal, {\em Rank and {{Geometry}} of {{Three Dimensional Matrices}}},
  1983, \url{https://three-mode.leidenuniv.nl/pdf/k/kruskal1983.pdf}.

\bibitem{kruskal_rank_1989}
J.~B. Kruskal, {\em Rank, {D}ecomposition, and {U}niqueness for 3-{W}ay and
  n-{W}ay {A}rrays}, in Multiway Data Analysis, North-Holland Publishing Co.,
  NLD, Oct. 1989, pp.~7--18.

\bibitem{landsberg_tensors_2012}
J.~M. Landsberg, {\em Tensors: {{Geometry}} and {{Applications}}}, vol.~128 of
  Graduate {{Studies}} in {{Mathematics}}, American Mathematical Society, 2012.

\bibitem{levin_effect_2025}
E.~Levin, J.~Kileel, and N.~Boumal, {\em The {E}ffect of {S}mooth
  {P}arametrizations on {N}onconvex {O}ptimization {L}andscapes}, Mathematical
  Programming, 209 (2025), pp.~63--111,
  \url{https://doi.org/10.1007/s10107-024-02058-3}.

\bibitem{nocedal_numerical_2006}
J.~Nocedal and S.~J. Wright, {\em Numerical Optimization}, Springer Series in
  Operations Research, Springer, New York, 2nd ed~ed., 2006.

\bibitem{paatero_construction_2000}
P.~Paatero, {\em Construction and {A}nalysis of {D}egenerate {{PARAFAC}}
  {M}odels}, Journal of Chemometrics, 14 (2000), pp.~285--299,
  \url{https://doi.org/10.1002/1099-128X(200005/06)14:3<285::AID-CEM584>3.0.CO;2-1}.

\bibitem{rudin_principles_1976}
W.~Rudin, {\em Principles of Mathematical Analysis}, International Series in
  Pure and Applied Mathematics, McGraw-Hill, New York, 3d~ed., 1976.

\bibitem{ruszczynski_nonlinear_2006}
A.~P. Ruszczy{\'n}ski, {\em Nonlinear Optimization}, Princeton University
  Press, Princeton, N.J, 2006.

\bibitem{seigal_real_2017}
A.~Seigal and B.~Sturmfels, {\em Real {R}ank {T}wo {G}eometry}, Journal of
  Algebra, 484 (2017), pp.~310--333,
  \url{https://doi.org/10.1016/j.jalgebra.2017.04.014}.

\bibitem{strassen_gaussian_1969}
V.~Strassen, {\em Gaussian {E}limination {I}s {N}ot {O}ptimal}, Numerische
  Mathematik, 13 (1969), pp.~354--356,
  \url{https://doi.org/10.1007/BF02165411}.

\bibitem{strassen_vermeidung_1973}
V.~Strassen, {\em Vermeidung von {{Divisionen}}}, Journal f{\"u}r die reine und
  angewandte Mathematik (Crelles Journal), 1973 (1973), pp.~184--202,
  \url{https://doi.org/10.1515/crll.1973.264.184}.

\bibitem{ten_berge_simplicity_1999}
J.~M. Ten~Berge and H.~A. Kiers, {\em Simplicity of {C}ore {A}rrays in
  {T}hree-{W}ay {P}rincipal {C}omponent {A}nalysis and the {T}ypical {R}ank of
  p{\texttimes}q{\texttimes}2 arrays}, Linear Algebra and its Applications, 294
  (1999), pp.~169--179, \url{https://doi.org/10.1016/S0024-3795(99)00057-9}.

\bibitem{ten_berge_kruskals_1991}
J.~M.~F. Ten~Berge, {\em Kruskal's {P}olynomial for 2{\texttimes}2{\texttimes}2
  {A}rrays and a {G}eneralization to 2{\texttimes}n{\texttimes}n {A}rrays},
  Psychometrika, 56 (1991), pp.~631--636,
  \url{https://doi.org/10.1007/BF02294495}.

\end{thebibliography}

\end{document}